\documentclass{amsart}
\usepackage{amssymb}
\usepackage{amsfonts}
\usepackage{amssymb}
\usepackage{amsmath}
\usepackage{amsthm}
\usepackage{enumerate}
\usepackage{tabularx}
\usepackage{centernot}
\usepackage{mathtools}
\usepackage{stmaryrd}
\usepackage{amsthm,amssymb}
\usepackage{etoolbox}
\usepackage{tikz}
\usepackage{amssymb}
\usetikzlibrary{matrix}
\usepackage{tikz-cd}

\usepackage{tikz}
\usetikzlibrary{decorations.pathmorphing,shapes}

\newcounter{sarrow}

\renewcommand{\leq}{\leqslant}
\renewcommand{\geq}{\geqslant}

\makeatletter
\def\subsection{\@startsection{subsection}{3}%
  \z@{.5\linespacing\@plus.7\linespacing}{.3\linespacing}%
  {\bfseries\centering}}
\makeatother

\makeatletter
\def\subsubsection{\@startsection{subsubsection}{3}%
  \z@{.5\linespacing\@plus.7\linespacing}{.3\linespacing}%
  {\centering}}
\makeatother

\makeatletter
\def\myfnt{\ifx\protect\@typeset@protect\expandafter\footnote\else\expandafter\@gobble\fi}
\makeatother

\theoremstyle{definition}

\newtheorem{theorem}{Theorem}[section]
\newtheorem{definition}[theorem]{Definition}
\newtheorem{lemma}[theorem]{Lemma}
\newtheorem{proposition}[theorem]{Proposition}
\newtheorem{example}[theorem]{Example}
\newtheorem{corollary}[theorem]{Corollary}
\newtheorem{oproblem}[theorem]{Open Problem}

\newtheorem{remark}[theorem]{Remark}

\newtheorem{definition/proposition}[theorem]{Definition/Proposition}

\newcounter{claimcounter}
\numberwithin{claimcounter}{theorem}
\newenvironment{claim}{\stepcounter{claimcounter}{\noindent {\bf Claim \theclaimcounter.}}}{}
\newenvironment{claimproof}[1]{\noindent{{\em Proof.}}\space#1}{\hfill $\rule{0.35em}{0.35em}$}

\newcommand{\pureindep}[1][]{%
  \mathrel{
    \mathop{
      \vcenter{
        \hbox{\oalign{\noalign{\kern-.3ex}\hfil$\vert$\hfil\cr
              \noalign{\kern-.7ex}
              $\smile$\cr\noalign{\kern-.3ex}}}
      }
    }\displaylimits_{#1}
  }
}

\newcommand{\indep}[2]{%
  \mathrel{
    \mathop{
      \vcenter{
        \hbox{%
\oalign{
\noalign{\kern-.3ex}\hfil$\vert$\hfil\cr
              \noalign{\kern-.7ex}
              $\smile$\cr\noalign{\kern-.3ex}
}
}
      }
}^{\!\!\!\!\!#2}_{\!\!\hspace{-0.1em}#1}
  }
}

\newcommand{\displayindep}[2]{%
  \mathrel{
    \mathop{
      \vcenter{
        \hbox{%
\oalign{
\noalign{\kern-.3ex}\hfil$\vert$\hfil\cr
              \noalign{\kern-.7ex}
              $\smile$\cr\noalign{\kern-.3ex}
}
}
      }
}^{\!\!\hspace{-0.1em}#2}_{\!\!\hspace{-0.1em}#1}
  }
}

\newcommand{\displayfindep}[2]{%
  \mathrel{
    \mathop{
      \vcenter{
        \hbox{%
\oalign{
\noalign{\kern-.3ex}\hfil$\vert$\hfil\cr
              \noalign{\kern-.7ex}
              $\smile$\cr\noalign{\kern-.3ex} 
}
}
      }
}^{\!\hspace{-0.14em}#2}_{\!\!\hspace{-0.05em}#1}
  }
}

\def\presuper#1#2%
  {\mathop{}%
   \mathopen{\vphantom{#2}}^{#1}%
   \kern-\scriptspace%
   #2}

\begin{document}

\begin{abstract} Based on Crapo's theory of one point extensions of combinatorial geometries, we find various classes of geometric lattices that behave very well from the point of view of stability  theory. One of them, $(\mathbf{K}^3, \preccurlyeq)$, is $\omega$-stable, it has a monster model and an independence calculus that satisfies all the usual properties of non-forking. On the other hand, these classes are rather unusual, e.g. in $(\mathbf{K}^3, \preccurlyeq)$ the Smoothness Axiom fails, and so $(\mathbf{K}^3, \preccurlyeq)$ is not an $\mathrm{AEC}$.
\end{abstract}

\title[Beyond $\mathrm{AECs}$: On the Model Theory of Geometric Lattices]{Beyond Abstract Elementary Classes: On the Model Theory of Geometric Lattices}
\thanks{The research of the second author was supported by the Finnish Academy of Science and Letters (Vilho, Yrj\"o and Kalle V\"ais\"al\"a foundation). The second author would like to thank John Baldwin and Will Boney for useful suggestions and discussions related to this paper.}

\author{Tapani Hyttinen}
\address{Department of Mathematics and Statistics,  University of Helsinki, Finland}

\author{Gianluca Paolini}
\address{Department of Mathematics and Statistics,  University of Helsinki, Finland}

\date{\today}
\maketitle


\section{Introduction}

	Not much is known about the model theory of pregeometries. Apart from some examples of geometries obtained from Hrushovski's constructions (see e.g. \cite{baldwin}), only a few sporadic studies appear in the literature (see e.g. \cite{shegirov}), and they are limited to a rather basic first-order analysis of classical geometries (mostly modular). In the present paper we look into this gap, in particular from a non-elementary point of view, i.e. from the perspective of {\em abstract elementary classes} \cite{shelah_abstr_ele_cla} ($\mathrm{AEC}$s for short).

	When looked upon as sets with a closure operator satisfying some additional requirements, pregeometries are not naturally seen as structures in the sense of model theory. On the other hand, when one moves from pregeometries to their associated canonical geometries, one gets (without loss of any essential information) objects that are in one-to-one correspondence with {\em geometric lattices}, which are natural ordered algebraic structures, and so they fit perfectly with the model-theoretic perspective.
	
	Assuming this perspective we also connect with the extensive literature on pseudoplanes and pseudospaces (see e.g. \cite{ample} and \cite{tent}). In this literature it is often assumed a graph-theoretic approach, where pseudospaces are considered as coloured graphs. In the case of geometric lattices the approach is of course lattice-theoretic, but the crucial {\em semimodularity} condition (one of the three defining them, see Definition \ref{def_geom_lat}) imposes that these structures are ranked, and thus naturally seen as coloured graphs. In fact, if the rank of the lattice is e.g. $3$, then geometric lattices are nothing but planes in the classical sense of geometry, i.e. systems of points and lines. On the other hand, the fact that we are working with semimodular lattices imposes much more geometric structure than in the case of pseudospaces, which are purely combinatorial in nature. For example, in planes each pair of lines intersect in at most one point, while in pseudoplanes lines can intersect in two different points. At the best of our knowledge, there is no single work on the model theory of semimodular lattice, and so looking at our work from this perspective we fill yet another gap in the model-theoretic literature. 

	In this paper we will always assume that our lattices are of finite rank (a.k.a. length) $n$, and in the most involved results we will further restrict to the case $n = 3$, i.e. planes (as already noticed). A point of divergence between our work and the standard references on the subject (e.g. \cite{aigner} and \cite{rota}) is that usually subgeometries are defined to be substructures in the vocabulary $L ' = \{ 0,1,\vee\}$ generated by a spanning set of atoms (this is of course motivated by various technical reasons). In this paper we consider also submodel relations in other vocabularies, specifically $L = \{ 0,1,\vee ,\wedge\}$, and various other strong submodel relations.

	The main question we ask is: are there good classes of geometric lattices? For a class to be good our minimum requirements are that the class should have a monster model and be stable. The most obvious classes fail to satisfy these requirements. Specifically, the class $\mathbf{K}^{n}_0$ of geometric lattices of a fixed rank $n \geq 3$, with the submodel relation $\preccurlyeq_{L'}$ in the vocabulary $L' = \{ 0,1,\vee\}$ as the strong submodel relation, does not satisfy the amalgamation property (Example \ref{failure_amalgamation}), and so it does not admit a monster model. 
	On the other hand, if one strengthens the vocabulary to $L = \{ 0,1,\vee ,\wedge\}$ and takes as the strong submodel relation the submodel relation $\preccurlyeq_{L}$ in this vocabulary, one does get amalgamation (Theorem \ref{amalgamation_finite_planes}). But one also gets all the bad stability-theoretic properties, as the monster model of $(\mathbf{K}^{n}_0, \preccurlyeq_{L})$ has the independence property (Theorem \ref{indepence_property}), and is thereby unstable. Restriction to modular lattices might help in restoring stability, but we are not interested in modular geometries here.

	So one needs to look elsewhere to find good classes of geometries. One classical  result on geometric lattices is Crapo's theorem on {\em one-point extensions} of combinatorial geometries \cite{crapo}. This result states that every finite geometric lattice of rank $n$ can be constructed from the Boolean algebra on $n$ atoms via a finite sequence of one-point extensions, which in turn are completely determined by so-called modular cuts (Definition \ref{def_mod_cut}).
	More than an operation on geometries, the theory of one-point extensions of geometric lattices is actually an explicit description of embeddings of finite geometries (in the sense of $L' = \{ 0,1,\vee\}$). Consequently, defining $\mathcal{A} \preccurlyeq \mathcal{B}$ to mean that $\mathcal{B}$ can be constructed from $\mathcal{A}$ via a sequence of one-point extensions, does not help. In fact, we are essentially back in the case $(\mathbf{K}^{n}_0, \preccurlyeq_{L'})$, which, as already noticed, is untractable according to our criteria.

	On the other hand, if we restrict to {\em principal} one-point extensions, i.e. extensions that correspond to principal modular cuts (Definition \ref{def_mod_cut}), then the situation is different. In fact, if one defines $\mathcal{A} \preccurlyeq \mathcal{B}$ to mean that $\mathcal{B}$ can be constructed from $\mathcal{A}$ via a sequence of principal one-point extensions and looks at the class $\mathbf{K}^{n}$ of structures in which the Boolean algebra on $n$ atoms $\preccurlyeq$-embeds, then $(\mathbf{K}^{n},\preccurlyeq)$ has a monster  model, and, at least for $n = 3$, it is stable in every infinite cardinality (Corollary \ref{stability}). Furthermore, $(\mathbf{K}^{3},\preccurlyeq)$ has a rather nice theory of independence, i.e. it admits what we call an independence calculus (Theorem \ref{good_frame_th}). On the other hand, the class is not an $\mathrm{AEC}$, as the Smoothness Axiom fails for $(\mathbf{K}^{3},\preccurlyeq)$ (Theorem \ref{th_fail_smooth}), although all the other axioms for $\mathrm{AEC}$ are satisfied (Theorem \ref{our_class_is_almost_AEC}).

	If we lift the well-foundedness requirement from $\preccurlyeq$, and further loosen up this relation so as to impose coherence, then we do get an $\mathrm{AEC}$ (denoted as $(\mathbf{K}_+^{3},\preccurlyeq^+)$) (Theorem \ref{+_is_AEC}), but it may not be $\omega$-stable any more. However, $(\mathbf{K}_+^{3},\preccurlyeq^+)$ remains at least stable (Theorem \ref{+_is_stable}).

\section{Geometric Lattices}\label{geo_lat}

	We begin with the definition of combinatorial geometries as closure systems of the form $(M, \mathrm{cl})$. Although our approach will be lattice-theoretic, we define (pre)geometric closure operators, in order to fix notations and definitions, and to state the well-known correspondence between finite dimensional combinatorial geometries and geometric lattices of finite rank. For undefined notions or details we refer the reader to \cite{aigner}, \cite{rota} and \cite{oxley}.
	
		\begin{definition}\label{def_pregeo} 
We say that $(M, \mathrm{cl})$ is a \emph{pregeometry} if for every $A, B \subseteq M$ and $a, b \in M$ the following conditions are satisfied:
	\begin{enumerate}[i)]
		\item $A \subseteq B$ implies $A \subseteq cl(A) \subseteq cl(B)$;
	\item if $a \in \mathrm{cl}(A \cup \left\{ b \right\}) - \mathrm{cl}(A)$, then $b \in \mathrm{cl}(A \cup \left\{ a \right\})$;
	\item if $a \in \mathrm{cl}(A)$, then $a \in \mathrm{cl}(A_0)$ for some finite $A_0 \subseteq A$.
\end{enumerate}
\end{definition} 

	\begin{example} Let $V$ be a finite dimensional vector space over a field $K$, and, for $A \subseteq V$, let $cl(A) = \langle A \rangle$ (linear span). Then $(V, cl)$ is a pregeometry.
\end{example}

\begin{definition} Let $(M, \mathrm{cl})$ be a pregeometry.
	\begin{enumerate}[i)]
	\item We say that $(M, \mathrm{cl})$ is a {\em combinatorial geometry} if $\mathrm{cl}(\emptyset) = \emptyset$ and $\mathrm{cl}(\left\{ p \right\}) = \left\{ p \right\}$, for every $p \in M$.
	\item We say that $(M, \mathrm{cl})$ is a {\em finite dimensional} pregeometry, a.k.a. {\em matroid}, if for every $B \subseteq M$ there exists finite $A \subseteq B$ with $\mathrm{cl}(A) = \mathrm{cl}(B)$. 
	\item We say that $(M, \mathrm{cl})$ is a {\em simple matroid} if it is a finite dimensional combinatorial geometry.
	\end{enumerate}
\end{definition}

	Recall that a lattice is an order $(L, \leq)$ such that any two elements $a, b$ have a least upper bound and a greatest lower bound, denoted by $a \wedge b$ and $a \vee b$, respectively. A chain in a lattice $(L, \leq)$ is a subset $X \subseteq L$ such that $(X, \leq)$ is a linear order.
	
	\begin{quote} {\bf Assumption: In this paper all lattices have a maximum element $1$ and a minimum element $0$. Furthermore, any chain between any two elements is finite.}
\end{quote}

	Given a lattice $(L, \leq)$ (as in the assumption above) and $x \in L$, we let $h(x)$, the height of $x$, to be the length of the longest maximal chain between $0$ and $x$. Furthermore, given $a, b \in L$, we say that $a$ is {\em covered} by $b$, for short $a \lessdot b$, if $a < b$ and for every $a \leq c \leq b$ we have that either $a = c$ or $c = b$. Finally, we say that $a$ is an {\em atom} if it covers $0$.
	
	\begin{definition}\label{def_geom_lat} Let $(L, \leq)$ be a lattice.
		\begin{enumerate}[i)]
		\item We say that $(L, \leq)$ is {\em semimodular} if for every $a, b \in L$ we have that 
		\[ a \wedge b \lessdot a \; \Rightarrow \; b \lessdot a \vee b. \]
		\item We say that $(L, \leq)$ is a {\em point lattice} if every $a \in L$ is a supremum of atoms.
		\item We say that $(L, \leq)$ is {\em geometric} if $(L, \leq)$ is a semimodular point lattice (without infinite chains).
\end{enumerate}
\end{definition}

	We now give an equivalent definition of geometric lattices that will be useful for our model-theoretic analysis of these structures.

	\begin{proposition}\label{equiv_def_geom_lat} A semimodular lattice $(L, \leq)$ (without infinite chains and with $0$ and $1$) is geometric if and only if it is relatively complemented, i.e. for every interval $[a, b]$ of $(L, \leq)$ and $c \in L$, there exists $d \in L$ such that $c \wedge d = a$ and $c \vee d = b$.
\end{proposition}

	\begin{proof} See e.g. \cite[Proposition 2.36]{aigner}.
\end{proof}

	It is possible to show that in a semimodular lattice any two maximal chains between $0$ and a fixed element $a$ have the same (finite) length. We denote this common length by $r(a)$, for rank (of course $r(a) = h(a)$). Furthermore, the atoms of a geometric lattice $(L, \leq)$ are called {\em points} and denoted by $P(L)$ or $P(L, \leq)$, if possible confusion arises.

	\begin{example}\label{FG_3} In Figure \ref{HasseFG_3} we see the Hasse diagram of the Boolean algebra of rank $3$. This lattice is the smallest geometric lattice of rank $3$, and it is denoted as $FG_3$. Throughout the paper we draw Hasse diagrams of geometric lattices omitting the straight lines representing the edge relation. Given that we always represent geometric lattices via their associated geometry and that the order is inclusion in the latter, this is with no loss of information.
	
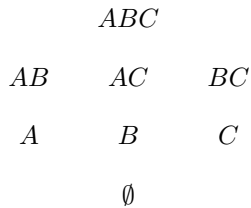
\begin{figure}[ht]
	\begin{center}
	\begin{tikzpicture}
\matrix (a) [matrix of math nodes, column sep=0.4cm, row sep=0.3cm]{
   & ABC &  \\
AB & AC  & BC \\
A  & B   & C  \\
& \emptyset \\};

\end{tikzpicture}
\end{center}  \caption{Hasse diagrams of $FG_3$.}\label{HasseFG_3}
\end{figure}
\end{example}

	The following well-known theorem establishes an explicit correspondence between geometric lattices and simple matroids, making geometric lattices one of the many cryptomorphic descriptions of matroids. 

	\begin{theorem}[Birkhoff-Whitney] 
	\begin{enumerate}[i)]
	\item Let $(M, \mathrm{cl})$ be a matroid and $L(M)$ the set of closed subsets of $M$, i.e. the $X \subseteq M$ such that $\mathrm{cl}(X) = X$. Then $(L(M), \subseteq)$ is a geometric lattice.
	\item Let $(L, \leq)$ be a geometric lattice with point set $M$ and for $A \subseteq M$ let $$\mathrm{cl}(A) = \left\{ p \in M \, | \, p \leq \bigvee A \right\}.$$ Then $(M, \mathrm{cl})$ is a simple matroid. Furthermore, the function $\phi: L \rightarrow L(M)$ such that $\phi(x) = \left\{ p \in M \, | \, p \leq x \right\}$ is a lattice isomorphism.
\end{enumerate}	
\end{theorem}

	\begin{proof} See e.g. \cite[Theorem 6.1]{aigner}.
\end{proof}

	We will often identify geometric lattices which are isomorphic through an isomorphism $\pi$ which is the identity on points.	This can be made formal passing from the given isomorphic lattices to the lattice of closed sets of their associated simple matroid, as in the theorem above. This should be kept in mind in what follows, in particular in Section \ref{plane}. We now introduce the notion of subgeometry and $\wedge$-subgeometry. The distinction between these two notions will be crucial in what follows.

	\begin{definition}\label{def_meet_subgeo} Let $(L, \leq)$ be a lattice.
		\begin{enumerate}[i)]
		\item Given $N \subseteq P(L)$, we define the subgeometry of $(L, \leq)$ generated by $N$ to be the $\vee$-semilattice of $(L, \leq)$ generated by $N$ together with the induced ordering.
		\item We say that $(L', \leq)$ is a {\em subgeometry} of $(L, \leq)$ if it is a subgeometry of $(L, \leq)$ generated by $N$ for some $N \subseteq P(L)$.
		\item We say that $(L', \leq)$ is a $\wedge$-{\em subgeometry} of $(L, \leq)$ if it is a subgeometry of $(L, \leq)$ and for every $a, b \in L'$ we have that $$(a \wedge b)^{L'} = (a \wedge b)^{L}.$$
\end{enumerate}
\end{definition}

\begin{example} Let $(L, \leq)$ be the geometry represented in the first diagram of Figure \ref{Hasse_meet_subgeo} (remember the convention about representations of geometric lattices explained in Example \ref{FG_3}). Let $(L', \leq)$ be the subgeometry generated by $\left\{ A, B, C, D \right\}$ (cf. the second diagram of Figure \ref{Hasse_meet_subgeo}) and $(L'', \leq)$ the subgeometry generated by $\left\{ A, B, C, E \right\}$ (cf. the third diagram of Figure \ref{Hasse_meet_subgeo}). Then $(L'', \leq)$ is a $\wedge$-subgeometry of $(L, \leq)$, but $(L', \leq)$ is not, in fact $(BCE \wedge ADE)^L = E \neq 0 = (BC \wedge AD)^{L'}$.
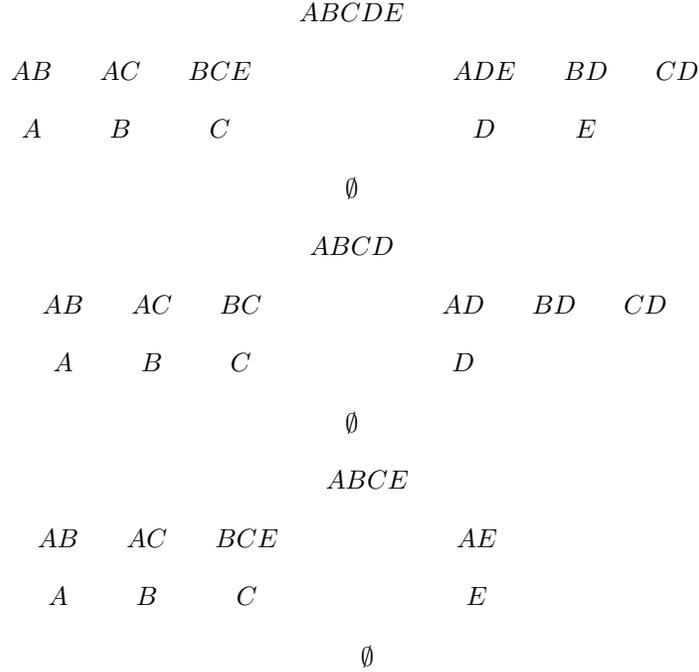
\begin{figure}[ht]
	\begin{center}
	\begin{tikzpicture}
\matrix (a) [matrix of math nodes, column sep=0.4cm, row sep=0.3cm]{
   &  & & ABCDE \\
AB & AC & BCE & & ADE & BD & CD \\
A  &  B  &  C & & D & E & \\
& & & \emptyset \\};
\end{tikzpicture}
	\begin{tikzpicture}
\matrix (a) [matrix of math nodes, column sep=0.4cm, row sep=0.3cm]{
   &  & & ABCD \\
AB & AC & BC & & AD & BD & CD \\
A  &  B  &  C & & D & & \\
& & & \emptyset \\};
\end{tikzpicture}
	\begin{tikzpicture}
\matrix (a) [matrix of math nodes, column sep=0.4cm, row sep=0.3cm]{
   &  & & ABCE \\
AB & AC & BCE & & AE & \phantom{AA} & \phantom{AA}\\
A  &  B  &  C & & E & \phantom{AA} & \\
& & & \emptyset \\};
\end{tikzpicture}
\end{center}  \caption{Subgeometry and $\wedge$-subgeometry.}\label{Hasse_meet_subgeo}
\end{figure}
\end{example}

	We now go into the theory of one-point extensions of geometric lattices \cite{crapo}. We first need some preliminary definitions.
	
	\begin{definition}\label{def_mod_cut} Let $(L, \leq)$ be a geometric lattice and $C \subseteq L$.
		\begin{enumerate}[i)]
		\item We say that $C$ is a {\em cut}, a.k.a. upperset, a.k.a. cone, if whenever $x \in C$ and $x \leq y \in L$, then $y \in C$.
		\item We say that $C$ is a {\em principal cut} if it is a cut of the form $\left\{ x \in L \, | \, a \leq x \right\}$ for some $a \in L$.
		\item We say that $C$ is a {\em modular cut} if it is a cut and for every $a, b \in C$, if $(a, b)$ is a modular pair, i.e. $$ r(a\vee b) + r(a \wedge b) = r(a) + r(b),$$ then $a \wedge b \in C$.
\end{enumerate}
\end{definition}

	Cuts containing points are called trivial. In our work we will consider only non-trivial cuts.

\begin{proposition}\label{one_point_1} Let $(L, \leq)$ be a geometric lattice with $P(L) = M \cup \left\{ p \right\}$ for $p \not\in M$, and $(L', \leq)$ the subgeometry of $(L, \leq)$ generated by $M$. Then the set $\left\{ x \in L' \, | \, p \leq x \right\}$ is a modular cut of the geometry $(L', \leq)$.
\end{proposition}

	More interestingly, also the converse of the proposition above is true.

	\begin{theorem}[Crapo \cite{crapo}]\label{one_point_2}  Let $C$ be a (non-trivial) modular cut of the geometric lattice $(L', \leq)$, and $p \not\in P(L')$. Then, modulo isomorphism of geometries, there is a unique geometry $(L, \leq)$ such that the following hold:
	\begin{enumerate}[i)]
	\item $P(L) = P(L') \cup \left\{ p \right\}$;
	\item $(L', \leq)$ is the subgeometry of $(L, \leq)$ generated by $P(L')$;
	\item $C = \left\{ x \in L' \, | \, p \leq x \right\}$.
	\end{enumerate}	
\end{theorem}

\begin{proof} See e.g. \cite[Theorem 6.53]{aigner}.
\end{proof}

	The geometry of the previous theorem will be denoted as $L' \oplus_C p = (L' \oplus_C p, \leq)$. We now give an explicit description of it. First a technical definition which will be handy. Given a geometric lattice $(L, \leq)$ and a modular cut $C$ of $L$, we define the {\em collar} of $C$ in $L$, for short $CO_L(C) = (CO_L(C), \leq)$, to be the subposet of $(L, \leq)$ consisting of those elements of $L$ which are not in $C$ and are not covered by an element of $C$.

	\begin{proposition} Let $(L', \leq)$, $C$ and $p$ as in the previous theorem. Then the geometric lattice $(L' \oplus_C p, \leq)$ is the lattice whose domain is the set $$L' \cup \left\{ a \vee p \, | \, a \in CO_{L'}(C)\right\},$$ and whose covering relation is defined according to the orders of the posets $(L', \leq)$ and $(CO_{L'}(C), \leq)$ and the following additional covering relations:
	\begin{enumerate}[i)]
	\item if $a \in CO_L(C)$, then $a \lessdot a \vee p$;
	\item if $a \in CO_L(C)$, $b \in C$ and $a \leq b$ with $r(b) = r(a) +2$, then $a \vee p \lessdot b$.
	\end{enumerate}
\end{proposition}

\begin{proof} See e.g. \cite[Proposition 6.54]{aigner}.
\end{proof}

	\begin{example}\label{affine_plane_order_three} The affine plane of order two, denoted as $AF(2)$, can be obtained from $FG_3$ (as in Example \ref{FG_3}) with a principal one-point extension, namely the one determined by the modular cut $\left\{ ABC \right\}$. 
\begin{figure}[ht]
	\begin{center}
	\begin{tikzpicture}
\matrix (a) [matrix of math nodes, column sep=0.4cm, row sep=0.3cm]{
   &  & & ABCD \\
AB & AC & BC & & AD & BD & CD \\
& A  & B & & C & D \\
& & & \emptyset \\};

\end{tikzpicture}
\end{center}  \caption{Hasse diagrams of $AF(2)$.}
\end{figure}
\end{example}	
	
	\begin{example}\label{example_Fano} The Fano plane, denoted as $PG(2, 2)$, can be obtained from $AF(2)$ with three non-principal one-point extensions. Starting from $AF(2)$ as in Example \ref{affine_plane_order_three}, first put the point $E$ into the lines $AB$ and $CD$, then $F$ into the lines $AC$ and $BD$ and finally $G$ into $BC$, $AD$ and $FE$.
\begin{figure}[ht]
	\begin{center}
	\begin{tikzpicture}
\matrix (a) [matrix of math nodes, column sep=0.4cm, row sep=0.3cm]{
   &  & & ABCDEFG \\
ABE & ACF & BCG & ADG & BDF & CDE & FEG \\
A  & B  & C & D & E & F & G \\
& & & \emptyset \\};
\end{tikzpicture}
\end{center}  \caption{Hasse diagrams of $PG(2, 2)$.}\label{Fano}
\end{figure}
\end{example}

	\begin{lemma}\label{princ_pres_ext} Let $C$ be a (non-trivial) modular cut of the geometric lattice $(L, \leq)$, and $p \not\in P(L)$. Then, $(L, \leq)$ is a $\wedge$-subgeometry of $(L \oplus_C p, \leq)$ if and only if $C$ is principal.	
\end{lemma}

\begin{proof} Routine.
\end{proof}

	If the modular cut $C$ is principal generated by $a \in L$, we denote $L \oplus_C p$ simply as $L \oplus_a p$. Principal extensions $L \oplus_a p$ are also referred to as adding a point freely under the closed set, a.k.a. flat, $a$. In the case of $a =1$ one simply says $L \oplus_a p$ has been obtained by adding a generic point to $L$.
	
	We conclude this section with a basic lemma which will be crucial in Section \ref{plane}.

	\begin{lemma}\label{switchining} Let $(L, \leq)$ be a geometric lattice and $a, b \in L$, with $r(a), r(b) \geq 2$. Then 
	\[ (L \oplus_a p_0) \oplus_b p_1 = (L \oplus_b p_1) \oplus_a p_0.\]
\end{lemma}

\begin{proof}  See e.g. \cite[Section 7.2, Exercise 8]{oxley}.
\end{proof}

\section{Abstract Elementary Classes}

	In this section we introduce the basics of abstract elementary classes (see e.g. \cite{shelah_abstr_ele_cla} and \cite{jarden}). This machinery will be used in Section \ref{plane} in order to study various classes of geometric lattices. As usual in this context, type means Galois type. Given a class $\mathbf{K}$ of structures in the vocabulary $L$, we denote by $\leq$ the $L$-submodel relation on structures in $\mathbf{K}$.

\begin{definition}[Abstract Elementary Class \cite{shelah_abstr_ele_cla}]\label{def_indep_first_order}  Let $\mathbf{K}$ be a class of structures in the vocabulary $L$ and $\preccurlyeq$ a binary relation on $\mathbf{K}$. We say that $(\mathbf{K}, \preccurlyeq)$ is an {\em abstract elementary class} ($\mathrm{AEC}$) if the following conditions are satisfied.
		\begin{enumerate}[(1)]
			\item $\mathbf{K}$ and $\preccurlyeq$ are closed under isomorphisms.
			\item If $\mathcal{A} \preccurlyeq \mathcal{B}$, then $\mathcal{A}$ is an $L$-submodel of $\mathcal{B}$ ($\mathcal{A} \leq \mathcal{B}$).
			\item The relation $\preccurlyeq$ is a partial order on $\mathbf{K}$.
			\item If $(\mathcal{A}_i)_{i < \delta}$ is an increasing continuous $\preccurlyeq$-chain, then:
			\begin{enumerate}[({4.}1)]
				\item $\bigcup_{i < \delta} \mathcal{A}_i \in \mathbf{K}$;
				\item for each $j < \delta$, $\mathcal{A}_j \preccurlyeq \bigcup_{i < \delta} \mathcal{A}_i$;
				\item if each $\mathcal{A}_j \preccurlyeq \mathcal{B}$, then $\bigcup_{i < \delta} \mathcal{A}_i \preccurlyeq \mathcal{B}$ \; (Smoothness Axiom).
	\end{enumerate}
			\item If $\mathcal{A}, \mathcal{B}, \mathcal{C} \in \mathbf{K}$, $\mathcal{A} \preccurlyeq \mathcal{C}$, $\mathcal{B} \preccurlyeq \mathcal{C}$ and $\mathcal{A} \leq \mathcal{B}$, then $\mathcal{A} \preccurlyeq \mathcal{B}$ \; (Coherence Axiom).
			\item There is a L\"owenheim-Skolem number $\mathrm{LS}(\mathbf{K}, \preccurlyeq)$ such that if $\mathcal{A} \in \mathbf{K}$ and $B \subseteq A$, then there is $\mathcal{C} \in \mathbf{K}$ such that $B \subseteq C$, $\mathcal{C} \preccurlyeq \mathcal{A}$ and $|C| \leq |B| + |L| + \mathrm{LS}(\mathbf{K}, \preccurlyeq)$ \; (Existence of LS-number).
	
\end{enumerate}
		
\end{definition}

	\begin{definition} If $\mathcal{A}, \mathcal{B} \in \mathbf{K}$ and $f: \mathcal{A} \rightarrow \mathcal{B}$ is an embedding such that $f(\mathcal{A}) \preccurlyeq \mathcal{B}$, then we say that $f$ is a $\preccurlyeq$-embedding.
	
\end{definition}
	
	Let $\lambda$ be a cardinal. We let $\mathbf{K}_{\lambda} = \left\{ \mathcal{A} \in \mathbf{K} \; | \; |A| = \lambda \right\}$.

	\begin{definition}\label{def_AP} Let $(\mathbf{K}, \preccurlyeq)$ be an $\mathrm{AEC}$. 
		\begin{enumerate}[(i)]
			\item We say that $(\mathbf{K}, \preccurlyeq)$ has the {\em amalgamation property} $(\mathrm{AP})$ if for any $\mathcal{A}, \mathcal{B}_0, \mathcal{B}_1 \in \mathbf{K}$ with $\mathcal{A} \preccurlyeq \mathcal{B}_i$ for $i < 2$, there are $\mathcal{C} \in \mathbf{K}$ and $\preccurlyeq$-embeddings $f_i: \mathcal{B}_i \rightarrow \mathcal{C}$ for $i < 2$, such that $f_0 \restriction A = f_1 \restriction A$.
			\item We say that $(\mathbf{K}, \preccurlyeq)$ has the {\em joint embedding property} $(\mathrm{JEP})$ if for any $\mathcal{B}_0, \mathcal{B}_1 \in \mathbf{K}$ there are $\mathcal{C} \in \mathbf{K}$ and $\preccurlyeq$-embeddings $f_i: \mathcal{B}_i \rightarrow \mathcal{C}$ for $i < 2$.
			\item We say that $(\mathbf{K}, \preccurlyeq)$ has arbitrarily large models $(\mathrm{ALM})$ if for every $\lambda \geq \mathrm{LS}(\mathbf{K}, \preccurlyeq)$, $\mathbf{K}_{\lambda} \neq \emptyset$.
\end{enumerate}
\end{definition}

	We say that $(\mathbf{K}, \preccurlyeq)$ is an {\em almost} $\mathrm{AEC}$ if it satisfies all the $\mathrm{AEC}$ axioms except possibly the Smoothness Axiom. Despite the failure of the Smoothness Axiom, for any almost $\mathrm{AEC}$, say $(\mathbf{K}, \preccurlyeq)$, with $\mathrm{AP}$, $\mathrm{JEP}$ and $\mathrm{ALM}$, it is still possible to construct a monster model $\mathfrak{M} = \mathfrak{M}(\mathbf{K}, \preccurlyeq)$ for $(\mathbf{K}, \preccurlyeq)$, i.e. a $\kappa$-model homogeneous and $\kappa$-universal (for $\kappa$ large enough) structure in $\mathbf{K}$.
The construction of the monster model $\mathfrak{M}(\mathbf{K}, \preccurlyeq)$, for $(\mathbf{K}, \preccurlyeq)$ an almost $\mathrm{AEC}$, is just as in the case of $\mathrm{AEC}$s. In fact, a careful inspection of the average construction of $\mathfrak{M}$ shows that the Smoothness Axiom is not needed.	

	Taking inspiration from Shelah's notion of a good $\lambda$-frame \cite[Definition 2.1.1]{jarden}, we now define a notion of an independence calculus $(\mathbf{K}, \preccurlyeq, \pureindep)$, for $(\mathbf{K}, \preccurlyeq)$ an almost $\mathrm{AEC}$  with $\mathrm{AP}$, $\mathrm{JEP}$ and $\mathrm{ALM}$. 
%
%
	In Section \ref{plane} we will define a particular class of planes $\mathbf{K} = \mathbf{K}^3$ and a notion of embedding $\preccurlyeq$, so that $(\mathbf{K}, \preccurlyeq)$ is an almost $\mathrm{AEC}$ with $\mathrm{AP}$, $\mathrm{JEP}$ and $\mathrm{ALM}$, and then show that it admits an independence calculus $(\mathbf{K}, \preccurlyeq, \pureindep)$.

		\begin{definition}\label{frame} We say that $(\mathbf{K}, \preccurlyeq, \pureindep)$ is an independence calculus if the following are satisfied.
\begin{enumerate}[(1)]
\item \begin{enumerate}[(a)]
	\item $(\mathbf{K}, \preccurlyeq)$ is an almost $\mathrm{AEC}$.
\end{enumerate}
    \item \begin{enumerate}[(a)]
	\item $(\mathbf{K}, \preccurlyeq)$ satisfies the joint embedding property.
	\item $(\mathbf{K}, \preccurlyeq)$ satisfies the amalgamation property.
	\item $(\mathbf{K}, \preccurlyeq)$ has arbitrarily large models.
\end{enumerate}
	\item The relation $\pureindep$ satisfies the following axioms. 
	\begin{enumerate}[(a)]
	\item $\pureindep$ is a set of triples $(a, \mathcal{A}_0,  \mathcal{A}_1)$ such that $\mathcal{A}_i \preccurlyeq \mathfrak{M}$, for $n = 0, 1$, and it respects isomorphisms, i.e. if $a \pureindep[\mathcal{A}_0] \mathcal{A}_1$ and $f \in \mathrm{Aut}(\mathfrak{M})$, then $f(a) \pureindep[f(\mathcal{A}_0)] f(\mathcal{A}_1)$.
	\item Symmetry. If $\mathcal{A}_0 \preccurlyeq \mathcal{A}_1 \preccurlyeq \mathfrak{M}$, $b \pureindep[\mathcal{A}_0] \mathcal{A}_1$ and $a \in A_1$, then there exists $\mathcal{A}_0 \preccurlyeq \mathcal{A}_2$ such that $b \in A_2$ and $a \pureindep[\mathcal{A}_0] \mathcal{A}_2$.
	\item Monotonicity. If $\mathcal{A} \preccurlyeq \mathcal{A}' \preccurlyeq \mathcal{B}' \preccurlyeq \mathcal{B} \preccurlyeq \mathfrak{M}$, and $a \pureindep[\mathcal{A}] \mathcal{B}$, then $a \pureindep[\mathcal{A}'] \mathcal{B}'$.
	\item Local Character. If $\delta$ is limit, $(\mathcal{A}_i)_{i \leq \delta}$ is an increasing continuous $\preccurlyeq$-chain and $\mathcal{A}_{\delta} \preccurlyeq \mathfrak{M}$, then there exists $\alpha < \delta$ such that $a \pureindep[\mathcal{A}_{\alpha}] \mathcal{A}_{\delta}$.
	\item Existence of non-forking extension. If $\mathcal{A} \preccurlyeq \mathcal{B} \preccurlyeq \mathfrak{M}$ and $a \in \mathfrak{M}^{< \omega}$, then there exists a $q \in S(\mathcal{B})$ such that $q \restriction \mathcal{A} = \mathrm{tp}(a/\mathcal{A})$ and for some $b \models q$ we have that $b \pureindep[\mathcal{A}] \mathcal{B}$.
	\item Uniqueness of the non-forking extension. If $\mathcal{A} \preccurlyeq \mathcal{B} \preccurlyeq \mathfrak{M}$, $a \pureindep[\mathcal{A}] \mathcal{B}$, $b \pureindep[\mathcal{A}] \mathcal{B}$ and $\mathrm{tp}(a/\mathcal{A}) = \mathrm{tp}(b/\mathcal{A})$, then $\mathrm{tp}(a/\mathcal{B}) = \mathrm{tp}(b/\mathcal{B})$.
	\item Continuity. If $\delta$ is limit, $(\mathcal{A}_i)_{i \leq \delta}$ is an increasing continuous $\preccurlyeq$-chain, $\mathcal{A}_{\delta} \preccurlyeq \mathfrak{M}$ and $a$, $(a_i)_{i < \delta}$ are such that $a_i \pureindep[\mathcal{A}_0] \mathcal{A}_i$ and $\mathrm{tp}(a_i/\mathcal{A}_i) = \mathrm{tp}(a/\mathcal{A}_i)$, then $a \pureindep[\mathcal{A}_0] \mathcal{A}_{\delta}$.
\end{enumerate}
\end{enumerate}
\end{definition}

	It is a well-known fact that good frames and independence calculi are a far-reaching generalization of the theory of linear independence. We conclude this section with a standard property of independence calculi which will be relevant in the proof of Theorem \ref{good_frame_th}.
	
	\begin{proposition}\label{transitivity}[Transitivity] Suppose that $(\mathbf{K}, \preccurlyeq, \pureindep)$ satisfies (1), (2) and (3)(a)-(f) of Definition \ref{frame} (i.e. everything except possibly (3)(g)), and let $\mathcal{A}_0 \preccurlyeq \mathcal{A}_1 \preccurlyeq \mathcal{A}_2 \preccurlyeq \mathfrak{M}$. Then,  $a \pureindep[\mathcal{A}_0] \mathcal{A}_1$ and $a \pureindep[\mathcal{A}_1] \mathcal{A}_2$ implies $a \pureindep[\mathcal{A}_0] \mathcal{A}_2$.
\end{proposition}

	\begin{proof} See \cite[II, Claim 2.18]{shelah_abstr_ele_cla}.
\end{proof}

\section{The Random Plane}\label{random_plane}

	As mentioned in the introduction, our starting questions were questions of amalgamation and stability. We asked: does the class of geometric lattices of a fixed rank $n$ have the amalgamation property in the vocabulary $L' = \left\{ 0, 1, \vee \right\}$? Does it have it in the vocabulary $L = \left\{ 0, 1, \vee, \wedge \right\}$? I.e. does this class of structures give rise to an amalgamation class with respect to the notion of subgeometry and $\wedge$-subgeometry, respectively? In this section we see that the answer to the first question is no, and we give a positive answer to the second question for $n = 3$ (i.e. planes). We then conclude that the class $(\mathbf{K}^{3}_{0}, \preccurlyeq_L)$ of geometric lattices of rank $3$ in the vocabulary $L$ is an $\mathrm{AEC}$ with $\mathrm{AP}$, $\mathrm{JEP}$ and $\mathrm{ALM}$. Finally, we show that $(\mathbf{K}^{3}_{0}, \preccurlyeq_L)$ is unstable.

	We now give an example of the failure of amalgamation for the class of geometric lattices of rank $n = 3$ in the vocabulary $L' = \left\{ 0, 1, \vee \right\}$. The example is taken from \cite[Example 7.2.3]{oxley}. Of course, the counterexample can be adapted so that we have failure for every $n \geq 3$.

\begin{example}\label{failure_amalgamation} Let $\mathcal{A}$ and $\mathcal{B}$ be the planes in Figure \ref{fail_amal} (lines of $\mathcal{A}$ in the first line, points of $\mathcal{A}$ in the second line, lines of $\mathcal{B}$ in the third line and points of $\mathcal{B}$ in the fourth line), and $\mathcal{C} = \langle A, B, C, D, E, F \rangle_{\mathcal{A}} \cong \langle A, B, C, D, E, F \rangle_{\mathcal{B}}$. Then there can not be a plane $\mathcal{D}$ in which both $\mathcal{A}$ and $\mathcal{B}$ $\preccurlyeq_{L'}$-embed over $\mathcal{C}$, because otherwise the lines $A \vee C$ and $B \vee E$ would have to intersect in two {\em different} points (the images of $P_0$ and $P_1$).
	
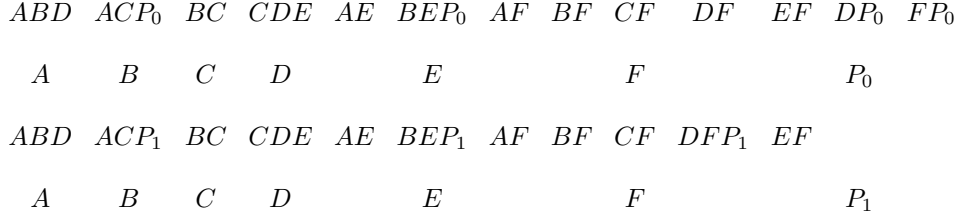
\begin{figure}[ht]
\begin{center}
\begin{tikzpicture}
\matrix (a) [matrix of math nodes, column sep=0.05cm, row sep=0.3cm]{
ABD & ACP_0 & BC & CDE & AE & BEP_0 & AF & BF & CF & DF & EF & DP_0 & FP_0 \\
A   & B  & C  & D  & & E & & & F & & & P_0 \\
ABD & ACP_1 & BC & CDE & AE & BEP_1 & AF & BF & CF & DFP_1 & EF  \\
A   & B  & C  & D  & & E & & & F & & & P_1 \\};
\end{tikzpicture}
\end{center}  \caption{Failure of amalgamation in $L' = \left\{ 0, 1, \vee \right\}$.} \label{fail_amal}
\end{figure}
\end{example}

	We now show that in the expanded vocabulary $L = \left\{ 0, 1, \vee, \wedge \right\}$ the class of planes {\em does} have the amalgamation property. The proof is essentially straightforward, this leads us to believe that this fact is know among experts.

\begin{theorem}\label{amalgamation_finite_planes} $(\mathbf{K}^3, \preccurlyeq_L)$ has $\mathrm{AP}$. 
\end{theorem}

\begin{proof} Let $\mathcal{A}, \mathcal{B}, \mathcal{C} \in \mathbf{K}^3$, with $C = A \cap B$ and $\mathcal{C} \leq \mathcal{A}, \mathcal{B}$. Let $(p_i)_{i < \alpha}$ be an injective enumeration of $P(\mathcal{B}) - P(\mathcal{C})$ and $\mathcal{B}_i = \langle C, (p_j)_{j < i} \rangle_{\mathcal{B}}$, for every $i < \alpha$. By induction on $i < \alpha$ we define $(\mathcal{D}_i)_{i < \alpha}$ such that
\begin{enumerate}[i)]
	\item $\mathcal{A}$ is a $\wedge$-subgeometry of $\mathcal{D}_i$;
	\item $\mathcal{B}_{i}$ a $\wedge$-subgeometry of $\mathcal{D}_i$;
	\item $\mathcal{D}_{i}$ is a subgeometry of $\mathcal{D}_{i+1}$.
\end{enumerate}
Of course $\bigcup_{i < \alpha} \mathcal{D}_i = \mathcal{D}$ will be the wanted amalgam of $\mathcal{A}$ and $\mathcal{B}$ over $\mathcal{C}$. If $i = 0$, let $\mathcal{D}_i = \mathcal{A}$. For $i$ limit, let $\mathcal{D}_i = \bigcup_{j < i} \mathcal{D}_j$. Suppose then that $i = j+1$ and let $M_i = \left\{ x \in B_{j} \, | \, \mathcal{B} \models p_{j} \leq x \right\}$.

\begin{claim} $M_i$ is a modular cut of $\mathcal{D}_{j}$.
\end{claim}

\begin{claimproof} By hypothesis $B_{j}$ is a $\wedge$-subgeometry of $\mathcal{D}_{j}$ and so the pair $(a, b)$ is modular in $B_{j}$ iff it is modular in $\mathcal{D}_{j}$.
\end{claimproof}

\noindent Then we can define $\mathcal{D}_i = \mathcal{D}_{j} \oplus_{M_i} p_{j}$. We verify the inductive properties. Item iii) is clear, we show ii). By inductive hypothesis, $\mathcal{B}_{j}$ is a subgeometry of $\mathcal{D}_{j}$ and furthermore
	\[ \left\{ x \in D_{i} \, | \, p_{j} \leq x \right\} = M_i = \left\{ x \in B_{i} \, | \, p_{j} \leq x \right\}, \]
and so $\mathcal{B}_{i}$ is a subgeometry of $\mathcal{D}_{i}$. We now show that $\mathcal{B}_{i}$ is actually a $\wedge$-subgeometry of $\mathcal{D}_{i}$. Let then $l_0, l_1$ be lines in $\mathcal{B}_{i}$, we want to show that $$(l_0 \wedge l_1)^{\mathcal{B}_{i}} = (l_0 \wedge l_1)^{\mathcal{D}_{i}}.$$ 
There are several cases.
\newline {\bf Case 1.} $l_0, l_1 \in B_{j}$.
\newline {\bf Case 1.1.} $l_0 \not\in M_i$ or $l_1 \not\in M_i$.
In this case we have (by the induction hypothesis that $B_j$ is a $\wedge$-subgeometry of $D_j$)
\[ \begin{array}{rcl}
		(l_0 \wedge l_1)^{\mathcal{B}_{i}} & = & (l_0 \wedge l_1)^{\mathcal{B}_{j}} \\
											& = & (l_0 \wedge l_1)^{\mathcal{D}_{j}} \\
											& = & (l_0 \wedge l_1)^{\mathcal{D}_{i}}.
		\end{array} \] 
		 {\bf Case 1.2.} $l_0, l_1 \in M_i$.
		 In this case we have 
	\[ (l_0 \wedge l_1)^{\mathcal{B}_{i}} = p_{j} = (l_0 \wedge l_1)^{\mathcal{D}_{i}}. \]
\newline {\bf Case 2.} $l_0 \not\in B_{j}$ (i.e it is a new line).
\newline {\bf Case 2.1.} $l_1 \not\in B_{j}$.
In this case we have 
\[ l_0 = q_0 \vee p_{j} \; \text{ and } \; l_1 = q_1 \vee p_{j}, \]
for $q_0, q_1 \in CO_{\mathcal{B}_{j}}(M_i) \subseteq CO_{\mathcal{D}_{j}}(M_i)$. Thus,
	\[ (l_0 \wedge l_1)^{\mathcal{B}_{i}} = p_{j} = (l_0 \wedge l_1)^{\mathcal{D}_{i}}. \]
\newline {\bf Case 2.2.} $l_1 \in B_{j}$. In this case we have that $l_0 = q_0 \vee p_{j}$ (as in the case above) and $l_1$ is an old line.
\newline {\bf Case 2.2.1.} $l_1 \in M_i$.
In this case we have 
\[ (l_0 \wedge l_1)^{\mathcal{B}_{i}} = p_{j} = (l_0 \wedge l_1)^{\mathcal{D}_{i}}. \]
\newline {\bf Case 2.2.2.} $l_1 \not \in M_i$. In this case we have 
\[ \begin{array}{rcl}
		(l_0 \wedge l_1)^{\mathcal{B}_{i}} = 0 & \text{ or } & (l_0 \wedge l_1)^{\mathcal{B}_{i}} = q_0 \\
		 & \Downarrow &  \\
		(l_0 \wedge l_1)^{\mathcal{D}_{i}} = 0 & \text{ or } & (l_0 \wedge l_1)^{\mathcal{D}_{i}} = q_0.
		\end{array} \] 
	Finally, we verify i). Of course $\mathcal{A}$ is a subgeometry of $\mathcal{D}_i$, because 
$\mathcal{A}$ is a subgeometry $\mathcal{D}_{j}$ and $\mathcal{D}_{j}$ is a subgeometry of $\mathcal{D}_{i}$. We show that it is actually a $\wedge$-subgeometry. Let $l_0, l_1$ be lines in $\mathcal{A}$. There are two cases.
\newline {\bf Case a.} $l_0, l_1 \in C$.
In this case we have (since $B_i$ is a $\wedge$-subgeometry of $D_i$)
\[ \begin{array}{rcl}
		(l_0 \wedge l_1)^{\mathcal{A}} & = & (l_0 \wedge l_1)^{\mathcal{C}} \\
											& = & (l_0 \wedge l_1)^{\mathcal{B}_{i}} \\
											& = & (l_0 \wedge l_1)^{\mathcal{D}_{i}}.
		\end{array} \] 
\newline {\bf Case b.} $l_0 \not\in C$ or $l_1 \not\in C$.
Suppose that $l_0$ and $l_1$ are parallel in $\mathcal{A}$ but incident in $\mathcal{D}_i$ in a point $p$. Then we must have that $p = p_{j}$ as $(l_0 \wedge l_1)^A = 0 = (l_0 \wedge l_1)^{D_j}$ by the inductive hypothesis. But then $l_0, l_1 \in M_i$, which is absurd, because $M_i \subseteq B_{j}$ and $B_{j} \cap A = C$.
\end{proof}

	The idea used in the proof above should generalize to $n > 3$, but unfortunately we were not able to establish this because of some combinatorial difficulties. We then leave open the question of amalgamation of geometric lattices or rank $n > 3$ in the vocabulary $L = \left\{ 0, 1, \vee, \wedge \right\}$. Once again, we suspect that this is known among combinatorialists. We denote by $\mathbf{K}^{n}_{0}$ the class of geometric lattices of fixed rank $n$.

\begin{oproblem} Does $(\mathbf{K}^{n}_{0}, \preccurlyeq_L)$ have $\mathrm{AP}$ for $n \geq 4$?
\end{oproblem}

	\begin{theorem} $(\mathbf{K}^{3}_{0}, \preccurlyeq_L)$ is an $\mathrm{AEC}$ with $\mathrm{AP}$, $\mathrm{JEP}$ and $\mathrm{ALM}$.	
\end{theorem}

	\begin{proof} Regarding $\mathrm{ALM}$, $\mathrm{JEP}$ and $\mathrm{AP}$, the first is obviously satisfied, the second follows from the third because $FG_3$ $\preccurlyeq_L$-embeds in every structure in $\mathbf{K}_0^{3}$, and the third is taken care of by Theorem \ref{amalgamation_finite_planes}. The axioms of $\mathrm{AEC}$s are obviously satisfied.
\end{proof}

	In (loose) analogy with the random graph, we call the monster model $\mathfrak{M}(\mathbf{K}^{3}_{0}, \preccurlyeq_L)$ the {\em random plane}.
	Notice that because of Proposition \ref{equiv_def_geom_lat} the class $\mathbf{K}^{n}_{0}$ is first-order axiomatizable, and in particular compact. We recall the definition of independence property for complete first-order theory $T$.

	\begin{definition} Let $T$ be a complete first-order theory with infinite models. We say that $T$ has the {\em independence property} if there exists a formula $\phi(x, y)$ such that for every $n < \omega$ there are $(a_i)_{i < n} \in (\mathfrak{M}^{|x|})^{n}$ and $(\vec{b}_J)_{J \subseteq n} \in (\mathfrak{M}^{|y|})^{2^n}$ such that 
	\[ \mathfrak{M} \models \phi(a_i, b_J) \;\; \text{ iff } \;\; i \in J. \]				
\end{definition}

\begin{theorem}\label{indepence_property} The first-order theory of the random plane has the independence property.
\end{theorem}

\begin{proof} Let $n < \omega$ and $\mathcal{A}$ be the geometry obtained by adding generic points $(p^i_j)_{j < n, \, i < 2}$ to $FG_3$ on point set $A, B, C$ (the order in which the $p^i_j$ are added does not matter, of course). Let then $a_i =  p_{i}^{0} \vee  p_{i}^{1}$, for $i < n$, and, for $J \subseteq n$, let $\mathcal{B}_J$ be the geometry obtained from $\mathcal{A}$ by first adding two generic points $(q_i)_{i < 2}$ and then adding points $p_i$ in the lines $a_i$ and $b_J = q_0 \vee q_1$, for all and only the $i \in J$. 
The structures $\mathcal{A}$ and $(\mathcal{B}_J)_{J \subseteq n}$ can of course be found in the random plane $\mathfrak{M}$, and so we have that for every $J \subseteq n$, $\mathfrak{M} \models a_i \wedge b_J \neq 0$ iff $i \in J$.
\end{proof}



	Notice that from the proof of Theorem \ref{indepence_property} it follows that $(\mathbf{K}^{3}_{0}, \preccurlyeq_L)$ is unstable (with respect to Galois types). Given the unstability of $(\mathbf{K}^{3}_{0}, \preccurlyeq_L)$, we searched for strengthening of the notion of $\wedge$-subgeometry that would still give rise to amalgamation. Based on the notion of principal one-point extension of Section \ref{geo_lat} and Lemma \ref{switchining}, we found one such notion: two geometric lattices of the same rank are strong one in another if the second can be obtained from the first by a sequence of {\em principal} one-point extensions (cf. the next section for a more thorough explanation).

\section{A Plane Beyond $\mathrm{AECs}$}\label{plane}

	In this and the next section we use the ideas mentioned at the end of Section \ref{random_plane} to study various classes of geometric lattices and notions of strong embeddings between them. We will focus on two notions of strong subgeometry, $\preccurlyeq$ and $\preccurlyeq^*$, and in function of them define two classes of planes, $\mathbf{K}^{3}$ and $\mathbf{K}^{3}_*$. We will see that $(\mathbf{K}^{3}, \preccurlyeq)$ is an almost $\mathrm{AEC}$ (Theorem \ref{our_class_is_almost_AEC}), it fails the Smootheness Axiom (Theorem \ref{th_fail_smooth}), and it is stable in every infinite cardinality (Corollary \ref{stability}). We then define an independence notion $\pureindep$ for $(\mathbf{K}^{3},  \preccurlyeq)$ and prove that $(\mathbf{K}^{3}, \preccurlyeq, \pureindep)$ is an independence calculus (Theorem \ref{good_frame_th}). Regarding $(\mathbf{K}^{3}_*, \preccurlyeq^*)$ we will see that there are difficulties in establishing the Coherence Axiom, but that if we modify $\preccurlyeq^*$ to a coherent relation $\preccurlyeq^+$, and accordingly extend $\mathbf{K}^{3}_*$ to a class $\mathbf{K}^{3}_+$, then $(\mathbf{K}^{3}_+, \preccurlyeq^+)$ is in fact an $\mathrm{AEC}$ (Theorem \ref{+_is_AEC}), and furthermore it is stable in every cardinal $\kappa$ such that $\kappa^{\omega} = \kappa$ (Theorem \ref{+_is_stable}). 
	
	We now use the theory of principal one-point extensions of geometric lattices to define a notion of construction between geometries which reminds of the notion of constructible sets for (abstract) isolation notion of classification theory, see e.g. \cite{shelah}. Given a set $X$ and $m \leq n < \omega$, we denote by $\mathcal{P}_{m/n}(X)$ the set of subsets of $X$ of cardinality at least $m$ and at most $n$. Recall that given a geometric lattice $\mathcal{A}$ we denote by $P(\mathcal{A})$ the set of points, i.e. rank $1$ objects, of $\mathcal{A}$. 
	
	\begin{definition}\label{construction} Let $\mathcal{A} \in \mathbf{K}^{n}_{0}$, $(a_i)_{i < \alpha}$ a sequence of elements so that $a_i \not\in A \cup \left\{ a_j \, | \, j < i \right\}$ and $(B_i \in \mathcal{P}_{2/n}(P(\mathcal{A}) \cup \left\{ a_j \, | \, j < i \right\})_{i < \alpha}$. By induction on $j < \alpha$, we define $\mathcal{A}_j \in \mathbf{K}^{n}_{0}$ as follows:
	\begin{enumerate}[i)]
	\item $j = 0$. $\mathcal{A}_0 = \mathcal{A}$;
	\item $j = i+1$. $\mathcal{A}_j = \mathcal{A}_i \oplus_{\bigvee B_i} a_i$, where $\bigvee B_i = (\bigvee B_i)^{\mathcal{A}_i}$;
	\item $j$ limit. $\mathcal{A}_j = \bigcup_{i < j} \mathcal{A}_i$.
\end{enumerate}
Finally, we define 
	$$\mathrm{CS}(\mathcal{A}, (a_i, B_i)_{i < \alpha}) = \bigcup_{i < \alpha} \mathcal{A}_i$$ 
($\mathrm{CS}$ is for construction).
\end{definition}

	I.e. $\mathcal{B} = \mathrm{CS}(\mathcal{A}, (a_i, B_i)_{i < \alpha})$ if $\mathcal{B}$ can be obtained from $\mathcal{A}$ via a sequence of principal one-point extensions (cfr. the notation introduced right after Lemma \ref{princ_pres_ext}). We refer to a sequence $(a_i, B_i)_{i < \alpha}$, as in the definitions above, as a  construction over $\mathcal{A}$. With some abuse of notation, we will refer to sequences of the form $(a_i, B_i)_{i \in X}$, for $X$ a set of ordinals, $a_i \not\in A \cup \left\{ a_j \, | \, j \in i \cap X \right\}$ and $(B_i \in \mathcal{P}_{2/n}(P(\mathcal{A}) \cup \left\{ a_j \, | \, j \in i \cap X \right\})_{i \in X}$, also as constructions over $\mathcal{A}$. We now define two notions of strong subgeometry, these will be the essential ingredients of this section. Notice that, for $\mathcal{A}, \mathcal{B} \in \mathbf{K}^{n}_{0}$,  $\mathcal{A}$ is a $\wedge$-subgeometry of $\mathcal{B}$ (cfr. Definition \ref{def_meet_subgeo}) iff $\mathcal{A}$ is an $L$-submodel of $\mathcal{B}$, for $L = \left\{ 0, 1, \vee, \wedge \right\}$.

	\begin{definition}\label{def_of_strong} Let $\mathcal{A}, \mathcal{B} \in \mathbf{K}^{n}_{0}$ with $\mathcal{A} \leq \mathcal{B}$, i.e. $\mathcal{A}$ is a $\wedge$-subgeometry of $\mathcal{B}$. 
	\begin{enumerate}[i)]
	\item We say that $\mathcal{A}$ is {\em strong} in $\mathcal{B}$, for short $\mathcal{A} \preccurlyeq \mathcal{B}$, if there exists a well-ordering $(a_i)_{i < \alpha}$ of $P(\mathcal{B}) - P(\mathcal{A})$ and $(B_i \in \mathcal{P}_{2/n}(P(\mathcal{A}) \cup \left\{ a_j \, | \, j < i \right\})_{i < \alpha}$ such that $$\mathcal{B} = \mathrm{CS}(\mathcal{A}, (a_i, B_i)_{i < \alpha}).$$
	\item We say that $\mathcal{A}$ is {\em almost strong} in $\mathcal{B}$, for short $\mathcal{A} \preccurlyeq^* \mathcal{B}$, if there exists a linear ordering $(a_i)_{i \in (I, <)}$ of $P(\mathcal{B}) - P(\mathcal{A})$ and $(B_i \in \mathcal{P}_{2/n}(P(\mathcal{A}) \cup \left\{ a_j \, | \, j < i \right\})_{i \in (I, <)}$ such that:
\begin{enumerate}[a)]
	\item for every $x \in B - A$ there exists $j \in I$ such that $x \in \langle A, (a_i)_{i \leq j}\rangle_{\mathcal{B}} - \langle A, (a_i)_{i < j}\rangle_{\mathcal{B}}$;
	\item for every $j \in I$ we have that
	$$\langle A, (a_i)_{i \leq j} \rangle_{\mathcal{B}} = \langle A, (a_i)_{i < j} \rangle_{\mathcal{B}} \oplus_{\bigvee B_i} a_j,$$
where $\bigvee B_i = (\bigvee B_i)^{\langle A, (a_i)_{i < j} \rangle_{\mathcal{B}}}$.
\end{enumerate}
\end{enumerate} 
\end{definition}

	The following proposition gives an equivalent characterization of $\preccurlyeq$ and $\preccurlyeq^*$.
	
	\begin{proposition}\label{equi_car_of_strong} Let $\mathcal{A}, \mathcal{B} \in \mathbf{K}^{n}_{0}$ with $\mathcal{A} \leq \mathcal{B}$.
	\begin{enumerate}[i)]
	\item $\mathcal{A} \preccurlyeq \mathcal{B}$ iff there exists a well-ordering $(a_i)_{i < \alpha}$ of $P(\mathcal{B}) - P(\mathcal{A})$ such that for every $j < \alpha$ we have that $\langle A, (a_i)_{i < j} \rangle_{\mathcal{B}} \leq \mathcal{B}$.
	\item $\mathcal{A} \preccurlyeq^* \mathcal{B}$ iff there exists a linear ordering $(a_i)_{i \in (I, <)}$ of $P(\mathcal{B}) - P(\mathcal{A})$ such that:
\begin{enumerate}[a)']
	\item for every 
$x \in B - A$  there exists $j \in I$ such that $x \in \langle A, (a_i)_{i \leq j}\rangle_{\mathcal{B}} - \langle A, (a_i)_{i < j}\rangle_{\mathcal{B}}$; 
	\item for every $J \subseteq I$ downward closed we have that $\langle A, (a_i)_{i \in J} \rangle_{\mathcal{B}} \leq \mathcal{B}$.
\end{enumerate}
\end{enumerate}
\end{proposition} 

	\begin{proof} It suffices to prove ii). The sufficiency of the condition b)' is clear. Since from b)' it follows in particular that for every $j \in I$ we have that $\langle A, (a_i)_{i < j} \rangle_{\mathcal{B}} \leq \mathcal{B}$, from which  $\langle A, (a_i)_{i < j} \rangle_{\mathcal{B}} \leq \langle A, (a_i)_{i \leq j} \rangle_{\mathcal{B}}$, because otherwise there would be $x, y \in \langle A, (a_i)_{i < j} \rangle_{\mathcal{B}}$ such that 
	$$r((x \wedge y)^{\langle A, (a_i)_{i < j} \rangle_{\mathcal{B}}}) < r((x \wedge y)^{\langle A, (a_i)_{i \leq j} \rangle_{\mathcal{B}}}) \leq r(x \wedge y)^{\mathcal{B}}),$$
	which contradicts the fact that $\langle A, (a_i)_{i < j} \rangle_{\mathcal{B}} \leq \mathcal{B}$. Thus, by Lemma \ref{princ_pres_ext}, for every $j \in I$ we have that $\langle A, (a_i)_{i \leq j} \rangle_{\mathcal{B}}$ is a principal extension of $\langle A, (a_i)_{i < j} \rangle_{\mathcal{B}}$, and so we can find the wanted $B_i$. We prove the necessity of the condition b)' under the assumption of condition a) = a)'. Suppose that there exists $J \subseteq I$ downward closed such that $\langle A, (a_i)_{i \in J} \rangle_{\mathcal{B}} \not\leq \mathcal{B}$. Then there exists $x, y \in \langle A, (a_i)_{i \in J} \rangle_{\mathcal{B}}$ such that $r((x \wedge y)^{\langle A, (a_i)_{i \in J} \rangle_{\mathcal{B}}}) < r((x \wedge y)^{\mathcal{B}})$. Let $t \in I$ be such that $(x \wedge y)^{\mathcal{B}} \in \langle A, (a_i)_{i \leq t}\rangle_{\mathcal{B}} - \langle A, (a_i)_{i < t}\rangle_{\mathcal{B}}$, then $t \not\in J$ because otherwise 
	$$r((x \wedge y)^{\langle A, (a_i)_{i \in J} \rangle_{\mathcal{B}}}) = r((x \wedge y)^{\langle A, (a_i)_{i \leq t} \rangle_{\mathcal{B}}}) = r((x \wedge y)^{\mathcal{B}}),$$
a contradiction.
Hence, being $J$ downward closed, for every $i \in J$ we have that $t > i$, and so $x, y \in \langle A, (a_i)_{i < t}\rangle_{\mathcal{B}}$. Thus, 
	$$(x \wedge y)^{\langle A, (a_i)_{i < t}\rangle_{\mathcal{B}}} \neq (x \wedge y)^{\mathcal{B}} = (x \wedge y)^{\langle A, (a_i)_{i \leq t} \rangle_{\mathcal{B}}},$$ 
and so again by Lemma \ref{princ_pres_ext} we are done, since then we cannot have that $\langle A, (a_i)_{i \leq t} \rangle_{\mathcal{B}}$ is a principal extension of $\langle A, (a_i)_{i < t} \rangle_{\mathcal{B}}$.
\end{proof}

	We then see that $\mathcal{A} \preccurlyeq \mathcal{B}$ implies $\mathcal{A} \preccurlyeq^* \mathcal{B}$ and, of course, if $P(\mathcal{B}) - P(\mathcal{A})$ is finite, then $\mathcal{A} \preccurlyeq \mathcal{B}$ iff $\mathcal{A} \preccurlyeq^* \mathcal{B}$.

\begin{remark}\label{technical_req} We explain the reasons behind the requirement a) in Definition \ref{def_of_strong}. Let $\preccurlyeq^*_{-}$ be the relation obtained from $\preccurlyeq^*$ by dropping condition a), $\mathcal{B}$ the rank $4$ geometry whose non-trivial lines (i.e. lines incident with at least three points) and hyperplanes are represented in Figure \ref{rank_4_ex}, $\mathcal{A} = \langle A, B, C, D, E_a = E, F_b = F, G_c = G \rangle_{\mathcal{B}}$ and $(I, <)$ the linear order $a < b < c < \cdots < -n < \cdots < -1 < 0$. Then $\mathcal{A} \preccurlyeq^*_{-} \mathcal{B}$ but the downward closed set $J = \left\{ a, b, c \right\} \subseteq I$ is such that $\langle A, (a_i)_{i \in J} \rangle_{\mathcal{B}} \not\leq \mathcal{B}$, in fact 
$$((A \vee B \vee G) \wedge (E \vee F \vee G))^{\mathcal{A}} = G \neq G \vee P_0 = ((A \vee B \vee G) \wedge (E \vee F \vee G))^{\mathcal{B}}.$$

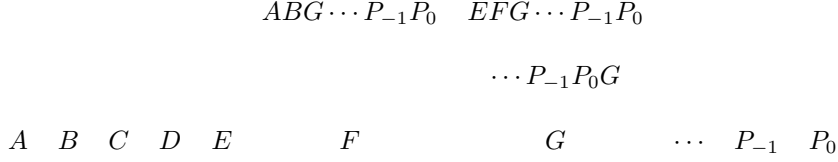
\begin{figure}[ht]
\begin{center}
\begin{tikzpicture}
\matrix (a) [matrix of math nodes, column sep=0.16cm, row sep=0.3cm]{
  &   &   &  & & ABG\cdots P_{-1}P_0 & EFG \cdots P_{-1}P_0 \\
  &   &   &  &    &   & \cdots P_{-1}P_0G  \\
A & B & C & D & E & F & G & \cdots & P_{-1} & P_0 & \\};
\end{tikzpicture}
\end{center} \caption{Reason I for requirement a).} \label{rank_4_ex} \end{figure}

\smallskip
\noindent
Another point against the relation $\preccurlyeq^*_{-}$ is that $\preccurlyeq^*_{-}$-witnesses need not determine isomorphism types, i.e. there are $\mathcal{B}, \mathcal{B}' \in \mathbf{K}_{0}^{3}$ such that $(a_i, B_i)_{i \in (I, <)}$ is a witness for both $\mathcal{A} \preccurlyeq^*_{-} \mathcal{B}$ and $\mathcal{A} \preccurlyeq^*_{-} \mathcal{B}'$, but $\mathcal{B} \not\cong \mathcal{B}'$. Let $\mathcal{B}$ and $\mathcal{B}'$ be the planes represented in Figure \ref{plane_for technical_req}, $\mathcal{A} = \langle A, B, C \rangle_{\mathcal{B}} = \langle A, B, C \rangle_{\mathcal{B}'}$, $(I, <)$ the linear order $\cdots < -n < \cdots < -1 < 0$ and $B_i = \left\{ P_{i-1}, P_{i-2} \right\}$. Then, in both $\mathcal{B}$ and $\mathcal{B}'$ the line $P_{-1} \vee P_0 = x$ is such that it does not exist $j \in I$ such that $x \in \langle A, (P_i)_{i \leq j}\rangle - \langle A, (P_i)_{i < j}\rangle$. Furthermore, $(P_i, B_i)_{i \in (I, <)}$ is a witness for both $\mathcal{A} \preccurlyeq_{-}^* \mathcal{B}$ and $\mathcal{A} \preccurlyeq_{-}^* \mathcal{B}'$, but obviously $\mathcal{B} \not\cong \mathcal{B}'$.
	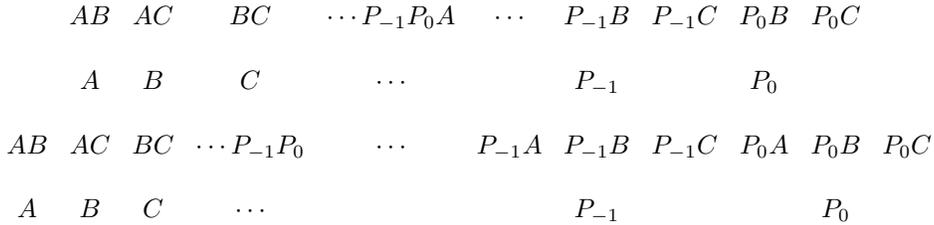
\begin{figure}[ht]
	\begin{center}
	\begin{tikzpicture}
\matrix (a) [matrix of math nodes, column sep=0.05cm, row sep=0.3cm]{
& AB & AC  & BC  & \cdots P_{-1} P_0 A & \cdots & P_{-1}B & P_{-1}C & P_{0}B & P_{0}C & \\
& A  & B   & C    &  \cdots  &  & P_{-1}  &  & P_0  & & \\
AB & AC  & BC  & \cdots P_{-1} P_0 & \cdots & P_{-1}A & P_{-1}B & P_{-1}C & P_{0}A & P_{0}B & P_{0}C & \\
A  & B   & C    &  \cdots  &  & & P_{-1}  &  & & P_0  & & \\};
\end{tikzpicture}
\end{center} \caption{Reason II for requirement a).} \label{plane_for technical_req} \end{figure}
\end{remark}

	Based on the two notions of embedding defined above, we define two classes of geometries of fixed rank $n$, the class of {\em principal geometries} of rank $n$ and the class of {\em quasi principal geometries} of rank $n$. 

\begin{definition}
\begin{enumerate}[i)]
\item The class $\mathbf{K}^{n}$ of {\em principal geometries} of rank $n$ is the class
	$$\left\{ \mathcal{A} \in \mathbf{K}^{n}_{0} \, | \, \exists \, \mathcal{B} \cong FG_n \text{ such that } \mathcal{B} \preccurlyeq \mathcal{A} \right\},$$
where $FG_n$ is the free geometry of rank $n$, i.e. the Boolean algebra on $n$ atoms.
\item The class $\mathbf{K}^{n}_{*}$ of {\em quasi principal geometries} of rank $n$ is the class
$$\left\{ \mathcal{A} \in \mathbf{K}^{n}_{0} \, | \, \exists \, \mathcal{B} \cong FG_n \text{ such that } \mathcal{B} \preccurlyeq^* \mathcal{A} \right\}.$$
\end{enumerate}
\end{definition}

	Unfortunately, we were not able to determine if $\mathbf{K}^{n} \neq \mathbf{K}^{n}_{*}$, but in Theorem \ref{th_fail_smooth} we will see that $(\mathbf{K}^{3}, \preccurlyeq) \neq (\mathbf{K}^{3}_{*}, \preccurlyeq^*)$. On the other hand, it is easy to see that $\mathbf{K}^{n} \neq \mathbf{K}^{n}_0$, e.g., for $n=3$, the Fano plane (Figure \ref{Fano}) is not in $\mathbf{K}^{3}$ because it can be obtained from $AF(2)$ with three non-principal one-point extensions (cf. Example \ref{example_Fano}). We will now focus on the study of $(\mathbf{K}^{n}, \preccurlyeq)$ and come back to the study of $(\mathbf{K}^{n}_*, \preccurlyeq^*)$ later on (only for $n = 3$ in this case). 

\begin{lemma}\label{finitary_change_order} Let $X$ and $Y$ be two finite sets of ordinals, $(a_i, B_i)_{i \in X}$ a construction over $\mathcal{A}$ and $\pi: Y \rightarrow X$ a bijection such that for every $j \in Y$, $B_{\pi(j)} \subseteq P(\mathcal{A}) \cup \left\{ a_{\pi(i)} \, | \, i \in j \cap Y \right\}$. Then
	$$\mathrm{CS}(\mathcal{A}, (a_i, B_i)_{i \in X}) = \mathrm{CS}(\mathcal{A}, (a_{\pi(i)}, B_{\pi(i)})_{i \in Y}) .$$
\end{lemma}

	\begin{proof} Let $X = \left\{ x_0 < \cdots < x_{n-1} \right\}$, $Y = \left\{ y_0 < \cdots < y_{n-1} \right\}$ and $w: n \rightarrow n$ so that $\pi(y_i) = x_{w_{(i)}}$. For $i < n$, let $a_{x_i} = c_i$ and $B_{x_i} = D_i$. We prove the lemma by induction on $|X| = n$. If $n \leq 1$ there is nothing to prove. Suppose then that $n \geq 2$.
\newline {\bf Case A.} $w(n-1) = n-1$. That is, $w \restriction n-1: n-1 \rightarrow n-1$. Thus, using the induction assumption we have that
\[ \begin{array}{rcl}
\mathrm{CS}(\mathcal{A}, (a_i, B_i)_{i \in X}) & = & 
\mathrm{CS}(\mathcal{A}, (c_{i}, D_{i})_{i < n}) \\
& = & \mathrm{CS}(\mathcal{A}, (c_i, D_i)_{i < n-1}) \oplus_{D_{n-1}} c_{n-1} \\ 
& = & \mathrm{CS}(\mathcal{A}, (c_{w(i)}, D_{w(i)})_{i < n-1} \oplus_{D_{n-1}} c_{n-1} \\ 
& = & \mathrm{CS}(\mathcal{A}, (c_{w(i)}, D_{w(i)})_{i < n} \\ 
& = & \mathrm{CS}(\mathcal{A}, (a_{\pi(i)}, B_{\pi(i)})_{i \in Y}).
\end{array} \] 
\newline {\bf Case B.} $w(n-1) \neq n-1$.
Of course to every $m \leq n$ and $X =  k_0 < \cdots < k_{m-1} \subseteq n$, it corresponds a bijection $t: m \rightarrow X$, namely $t(i) = k_{(i)}$. Let then
\begin{enumerate}[1)]
\item $(n, <)$ be the linear order $w(0) < \cdots < w(n-1)$;
\item $\sigma: n-1 \rightarrow n-1$ the bijection corresponding to $(n - \left\{ n-1 \right\}, <)$;
\item $\tau: n-1 \rightarrow n - \left\{ \sigma(n-2) \right\}$ the bijection corresponding to $(n - \left\{ \sigma(n-2), n-1 \right\}, <)$ concatenated with $\left\{ n-1 \right\}$, i.e. with $n-1$ as maximum;
\item $\rho: n-1 \rightarrow n - \left\{ \sigma(n-2) \right\}$ the bijection corresponding to $(n - \left\{ \sigma(n-2) \right\}, <)$.
\end{enumerate}
Notice that $\sigma(n-2) = w(n-1) \neq n-1$. Now, using the induction assumption twice and Lemma \ref{switchining} we have that
\[ \begin{array}{rcl}
\mathrm{CS}(\mathcal{A}, (a_i, B_i)_{i \in X}) & = & 
\mathrm{CS}(\mathcal{A}, (c_{i}, D_{i})_{i < n}) \\
& = & \mathrm{CS}(\mathcal{A}, (c_i, D_i)_{i < n-1}) \oplus_{D_{n-1}} c_{n-1} \\ 
& = & \mathrm{CS}(\mathcal{A}, (c_{\sigma(i)}, D_{\sigma(i)})_{i < n-1} \oplus_{D_{n-1}} c_{n-1} \\ 
& = & \mathrm{CS}(\mathcal{A}, (c_{\sigma(i)}, D_{\sigma(i)})_{i < n-2} \oplus_{D_{\sigma(n-2)}} c_{\sigma(n-2)} \oplus_{D_{n-1}} c_{n-1} \\ 
& = & \mathrm{CS}(\mathcal{A}, (c_{\sigma(i)}, D_{\sigma(i)})_{i < n-2} \oplus_{D_{n-1}} c_{n-1} \oplus_{D_{\sigma(n-2)}} c_{\sigma(n-2)} \\ 
& = & \mathrm{CS}(\mathcal{A}, (c_{\tau(i)}, D_{\tau(i)})_{i < n-1} \oplus_{D_{\sigma(n-2)}} c_{\sigma(n-2)} \\ 
& = & \mathrm{CS}(\mathcal{A}, (c_{\rho(i)}, D_{\rho(i)})_{i < n-1} \oplus_{D_{\sigma(n-2)}} c_{\sigma(n-2)} \\ 
& = & \mathrm{CS}(\mathcal{A}, (c_{w(i)}, D_{w(i)})_{i < n} \\ 
& = & \mathrm{CS}(\mathcal{A}, (a_{\pi(i)}, B_{\pi(i)})_{i \in Y}).
\end{array} \] 
\end{proof}

\begin{definition} Let $(a_i, B_i)_{i < \alpha}$ be a construction over $\mathcal{A}$ and $X \subseteq \alpha$. We say that $X$ is closed if for every $j \in X$, $B_{j} \subseteq P(\mathcal{A}) \cup \left\{ a_{i} \, | \, i \in j \cap X \right\}$. 
\end{definition}

\begin{lemma}\label{Konig} Let $(a_i, B_i)_{i < \alpha}$ be a construction over $\mathcal{A}$ and $X' \subseteq \alpha$. Then there is closed $X \subseteq \alpha$ such that $X' \subseteq X$ and $|X|$ is finite if $X'$ is finite, and $|X| = |X'|$ otherwise.
\end{lemma}

\begin{proof} K\"onig's Lemma.
\end{proof}

	We need a small technical proposition.

\begin{proposition}\label{closed_implies_subgeo}  Let $(a_i, B_i)_{i < \alpha}$ be a construction over $\mathcal{A}$, $X \subseteq \alpha$ finite and closed, and set $\mathrm{CS}(\mathcal{A}, (a_i, B_i)_{i < \alpha}) = \mathcal{B}$. Then 
	$$\mathrm{CS}(\mathcal{A}, (a_i, B_i)_{i \in X}) = \langle A, (a_{i})_{i \in X }\rangle_{\mathcal{B}}.$$
\end{proposition}

\begin{proof} Let $X = \left\{ k_0 < \cdots < k_{n-1} \right\}$. By induction on $j \leq n$ we we show that 
	$$\mathrm{CS}(\mathcal{A}, (a_{k_i}, B_{k_i})_{i < j}) = \langle A, (a_{k_i})_{i < j }\rangle_{\mathcal{B}}.$$
$j = 0$). $\mathrm{CS}(\mathcal{A}, (a_{k_i}, B_{k_i})_{i < j}) = \mathcal{A} = \langle A \rangle_{\mathcal{B}}$. 
\newline $j = i+1$). By construction,
	$$\mathcal{A}_{k_i + 1} = \mathcal{A}_{k_i} \oplus_{\bigvee B_{k_i}} a_{k_i},$$
where the notation is as in Definition \ref{construction}.
Now,
	$$M = \left\{ x \in A_{k_i} \, | \, a_{k_i} \leq x \right\} = \left\{ x \in A_{k_i} \, | \, (\bigvee B_{k_i})^{\mathcal{B}} \leq x \right\}.$$
By the coherence of the subgeometry relation,
	$$ \mathcal{A}^* = \langle A, (a_{k_0}, ..., a_{k_{i-1}}) \rangle_{\mathcal{B}}$$
is a subgeometry of $\mathcal{A}_{k_i}$.
Thus,
	$$M' = \left\{ x \in A^* \, | \, a_{k_i} \leq x \right\} \supseteq M. $$
Furthermore, by the fact that $X$ is closed, $B_{k_i} \subseteq \left\{ a_t \, | \, t \in k_i \cap X \right\}$. Hence, 
	$$M' = \left\{ x \in A^* \, | \, (\bigvee B_{k_i})^{\mathcal{A}^*} = (\bigvee B_{k_i})^{\mathcal{A}_{k_i}} = (\bigvee B_{k_i})^{\mathcal{B}} \leq x \right\},$$
and so $M'$ is principal cut of $A^*$.
Finally,
	\[ \begin{array}{rcl}
		\langle A, (a_{k_0}, ..., a_{k_{i-1}}, a_{k_{i}}) \rangle_{\mathcal{B}} & = & \langle A, (a_{k_0}, ..., a_{k_{i-1}}) \rangle_{\mathcal{B}} \oplus_{M'} a_{k_i} \\
						& = & \langle A, (a_{k_0}, ..., a_{k_{i-1}}) \rangle_{\mathcal{B}} \oplus_{\bigvee B_{k_i}} a_{k_i} \\		 				
		 				& = & \mathrm{CS}(\mathcal{A}, (a_{k_0}, B_{k_0}), ..., (a_{k_{i-1}}, B_{k_{i-1}})) \oplus_{\bigvee B_{k_i}} a_{k_i} \\
						& = & \mathrm{CS}(\mathcal{A}, (a_{k_0}, B_{k_0}), ..., (a_{k_i}, B_{k_i})) .\end{array} \] 
\end{proof}

	We now have all the ingredients to prove an infinitary version on Lemma \ref{finitary_change_order}.
	
\begin{lemma}\label{change_order} Let $(a_i, B_i)_{i < \alpha}$ be a construction over $\mathcal{A}$ and $\pi: \beta \rightarrow \alpha$ a bijection such that for every $j < \beta$, $B_{\pi(j)} \subseteq P(\mathcal{A}) \cup \left\{ a_{\pi(i)} \, | \, i < j \right\}$. Then
	$$\mathrm{CS}(\mathcal{A}, (a_i, B_i)_{i < \alpha}) = \mathrm{CS}(\mathcal{A}, (a_{\pi(i)}, B_{\pi(i)})_{i < \beta}) .$$
\end{lemma}

\begin{proof} Suppose not. By symmetry, we can assume that there is $I_0 \subseteq_{\omega} P(\mathcal{A}) \cup \left\{ a_i \, | \, i < \alpha \right\}$ such that $I_0$ is independent in $\mathrm{CS}(\mathcal{A}, (a_i, B_i)_{i < \alpha})$ but dependent in $\mathrm{CS}(\mathcal{A}, (a_{\pi(i)}, B_{\pi(i)})_{i < \beta})$. Let $X \subseteq_{\omega} \alpha$ be closed and such such that $I_0 \subseteq P(\mathcal{A}) \cup \left\{ a_i \, | \, i \in X \right\}$, this is possible by Lemma \ref{Konig}.

\begin{claim} For every $j \in \pi^{-1}(X)$, $B_{\pi(j)} \subseteq P(\mathcal{A}) \cup \left\{ a_{\pi(i)} \, | \, i \in j \cap \pi^{-1}(X) \right\}$. I.e. $\pi^{-1}(X) \subseteq \beta$ is closed. 
\end{claim}

\begin{claimproof} Standard.
\end{claimproof}

\noindent Let $\pi'= \pi\restriction{\pi^{-1}(X)}: \pi^{-1}(X) \rightarrow X$. By the fact that $X$ is closed, $\mathrm{CS}(\mathcal{A}, (a_i, B_i)_{i \in X})$ is well-defined, and by the claim $\pi'$ is as in Lemma \ref{finitary_change_order}. Thus, 
\begin{equation*}
\begin{tikzcd}[column sep=6pc]
\mathrm{CS}(\mathcal{A}, (a_i, B_i)_{i \in X}) \arrow{r}{=} \arrow{d}{\subseteq} &
 \mathrm{CS}(\mathcal{A}, (a_{\pi(i)}, B_{\pi(i)})_{i \in \pi^{-1}(X)}) \arrow{d}{\subseteq} \\
\mathrm{CS}(\mathcal{A}, (a_i, B_i)_{i < \alpha})  & \mathrm{CS}(\mathcal{A}, (a_{\pi(i)}, B_{\pi(i)})_{i < \beta}),
\end{tikzcd}
\end{equation*}
where the vertical arrows exist by Proposition \ref{closed_implies_subgeo} (as already noticed not only $X$ is closed, but also $\pi^{-1}(X)$ is).  But this is obviously a contradiction, since then $I_0$ is at the same time independent and dependent in $\mathrm{CS}(\mathcal{A}, (a_i, B_i)_{i \in X}) = \mathrm{CS}(\mathcal{A}, (a_{\pi(i)}, B_{\pi(i)})_{i \in \pi^{-1}(X)})$, being it independent in $\mathrm{CS}(\mathcal{A}, (a_i, B_i)_{i < \alpha})$ and dependent in $\mathrm{CS}(\mathcal{A}, (a_{\pi(i)}, B_{\pi(i)})_{i < \beta})$.
\end{proof}

	We prove an improved version of Proposition \ref{closed_implies_subgeo}.

\begin{lemma}\label{closed_implies_strong_subgeo} Let $(a_i, B_i)_{i < \alpha}$ be a construction over $\mathcal{A}$, $X \subseteq \alpha$ closed and set $\mathrm{CS}(\mathcal{A}, (a_i, B_i)_{i < \alpha}) = \mathcal{B}$. Then 
	$$\mathcal{A} \preccurlyeq \mathrm{CS}(\mathcal{A}, (a_i, B_i)_{i \in X}) = \langle A, (a_{i})_{i \in X }\rangle_{\mathcal{B}} \preccurlyeq \mathcal{B}.$$
\end{lemma}

\begin{proof} Let $\beta$ be the order type of $(X, \in)$ and $\gamma$ the order type of $(\alpha - X, \in)$. Let $\delta = \beta + \gamma$ and $\pi: \delta \rightarrow \alpha$ such that for all $i < \beta$, $\pi(i)$ is the $i$-th member of $(X, \in )$, and for all $i < \gamma$, $\pi(\beta + i)$ is the $i$-th member of $(\alpha - X, \in)$. It is then easy to see that for every $j < \delta$, $B_{\pi(j)} \subseteq P(\mathcal{A}) \cup \left\{ a_{\pi(i)} \, | \, i < j \right\}$. Hence, by Lemma \ref{change_order}, 
$$\mathrm{CS}(\mathcal{A}, (a_i, B_i)_{i < \alpha}) = \mathrm{CS}(\mathcal{A}, (a_{\pi(i)}, B_{\pi(i)})_{i < \delta}) .$$
But
\[ \begin{array}{rcl}
 \mathrm{CS}(\mathcal{A}, (a_i, B_i)_{i \in X})
 & = & \langle A, (a_{i})_{i \in X }\rangle_{\mathrm{CS}(\mathcal{A}, (a_{\pi(i)}, B_{\pi(i)})_{i < \beta})} \\
  & = & \langle A, (a_{i})_{i \in X }\rangle_{\mathrm{CS}(\mathcal{A}, (a_{\pi(i)}, B_{\pi(i)})_{i < \delta})} \\
 & \preccurlyeq & \mathrm{CS}(\mathcal{A}, (a_{\pi(i)}, B_{\pi(i)})_{i < \beta}) \\
 & \preccurlyeq & \mathrm{CS}(\mathcal{A}, (a_{\pi(i)}, B_{\pi(i)})_{i < \delta}),
\end{array} \] 
and so $\mathcal{A} \preccurlyeq \mathrm{CS}(\mathcal{A}, (a_i, B_i)_{i \in X}) = \langle A, (a_{i})_{i \in X }\rangle_{\mathcal{B}} \preccurlyeq \mathcal{B}$.
\end{proof}

	We now define the {\em free $\preccurlyeq$-amalgam} of $\mathcal{A}$ and $\mathcal{B}$ over $\mathcal{C}$, in symbols $\mathcal{A} \oplus_{\mathcal{C}} \mathcal{B}$.
	
	\begin{definition}\label{free_amalgamation} Let $\mathcal{A}, \mathcal{B}, \mathcal{C} \in \mathbf{K}^n$, with $\mathcal{C} \preccurlyeq \mathcal{A}, \mathcal{B}$ and $A \cap B = C$, and $\mathcal{A} = \mathrm{CS}(\mathcal{C}, (a_i, A_i)_{i < \alpha})$ and $\mathcal{B} = \mathrm{CS}(\mathcal{C}, (b_i, B_i)_{i < \beta})$. Let $\delta = \alpha + \beta$ and for $i < \delta$ set
	\[c_i = \begin{cases} a_i \;\;\;\;\; \text{ if } i < \alpha \\
						  b_j \;\;\;\;\; \text{ if } i = \alpha + j \text{ for } j < \beta \end{cases}
	\;\;\;\;\;\;\;  C_i = \begin{cases} A_i \;\;\;\;\; \text{ if } i < \alpha \\
						  B_j \;\;\;\;\; \text{ if } i = \alpha + j \text{ for } j < \beta. \end{cases}\]
We then define $\mathcal{A} \oplus_{\mathcal{C}} \mathcal{B} = \mathrm{CS}(\mathcal{C}, (c_i, C_i)_{i < \delta})$.
\end{definition}

	\begin{proposition}\label{amalg_strong} Let $\mathcal{A}, \mathcal{B}, \mathcal{C}$ as above. Then $\mathcal{A} \preccurlyeq \mathcal{A} \oplus_{\mathcal{C}} \mathcal{B} \in \mathbf{K}^n$ and $\mathcal{A} \oplus_{\mathcal{C}} \mathcal{B} = \mathcal{B} \oplus_{\mathcal{C}} \mathcal{A}$. In particular, $\mathcal{A}, \mathcal{B} \preccurlyeq \mathcal{A} \oplus_{\mathcal{C}} \mathcal{B}$.
\end{proposition}

	\begin{proof} Obviously $\mathcal{A} \preccurlyeq \mathcal{A} \oplus_{\mathcal{C}} \mathcal{B}$. The fact that $\mathcal{A} \oplus_{\mathcal{C}} \mathcal{B} \in \mathbf{K}^n$ follows from Transitivity of $\preccurlyeq$. Finally, to see that $\mathcal{A} \oplus_{\mathcal{C}} \mathcal{B} = \mathcal{B} \oplus_{\mathcal{C}} \mathcal{A}$, let $\mathcal{A} = \mathrm{CS}(\mathcal{C}, (a_i, A_i)_{i < \alpha})$, $\mathcal{B} = \mathrm{CS}(\mathcal{C}, (b_i, B_i)_{i < \beta})$ and $\pi: \beta + \alpha \rightarrow \alpha + \beta$ the obvious bijection, i.e.
	\[\pi(i) = \begin{cases} \alpha + i \;\;\;\; \text{ if } i < \beta \\
						     j \;\;\;\;\;\;\;\;\;\;\, \text{ if } i = \beta + j \text{ for } j < \alpha. \end{cases}\]
Then $\pi$ is as in Lemma \ref{change_order}, and so
\[ \mathrm{CS}(\mathcal{C}, (c_i, C_i)_{i < \alpha + \beta}) = \mathrm{CS}(\mathcal{C}, (c_{\pi(i)}, C_{\pi(i)})_{i < \beta + \alpha}). \]
\end{proof}

\begin{theorem}\label{our_class_is_almost_AEC} $(\mathbf{K}^{3}, \preccurlyeq)$ is an almost $\mathrm{AEC}$ with $\mathrm{AP}$, $\mathrm{JEP}$ and $\mathrm{ALM}$. 
\end{theorem}

\begin{proof} (1), (2) and (3) from Definition \ref{def_indep_first_order} are OK (for (3) just concatenate the constructions). Regarding (4), we will see that (4.3) fails in Theorem \ref{th_fail_smooth}, while (4.1) and (4.2) are taken care of by the limit clause in Definition \ref{construction}. Regarding $\mathrm{ALM}$, $\mathrm{JEP}$ and $\mathrm{AP}$, the first is obviously satisfied, the second follows from the third because $FG_3$ $\preccurlyeq$-embeds in every structure in $\mathbf{K}^{3}$, and the third is taken care of by the free $\preccurlyeq$-amalgam $\mathcal{A} \oplus_{\mathcal{C}} \mathcal{B}$. Items (5) and (6) need proofs. First (5). Let $\mathcal{A}, \mathcal{B} \preccurlyeq \mathcal{C}$ and $\mathcal{A} \leq \mathcal{B}$, we want to show that $\mathcal{A} \preccurlyeq \mathcal{B}$. Using the equivalent definition of $\preccurlyeq$ of Proposition \ref{equi_car_of_strong}, let $(a_i)_{i < \alpha}$ be a witness for $\mathcal{A} \preccurlyeq \mathcal{C}$, i.e. $(a_i)_{i < \alpha}$ is a well-ordering of $P(\mathcal{C}) - P(\mathcal{A})$ such that for every $j < \alpha$ we have that $\langle A, (a_i)_{i < j} \rangle_{\mathcal{C}}$ is a $\wedge$-subgeometry of $\mathcal{C}$.  Let $(a_i)_{i \in J \subseteq \alpha}$ be the well-order of $P(\mathcal{B}) - P(\mathcal{A})$ induced by $(a_i)_{i < \alpha}$. We show that $(a_i)_{i \in J}$ is a witness for $\mathcal{A} \preccurlyeq \mathcal{B}$. For a contradiction, suppose that there exists $j \in J$ such that there exists $x, y \in \langle A, (a_i)_{i \in J, \, i < j} \rangle_{\mathcal{B}} = \mathcal{A}_j$ such that $(x \wedge y)^{\mathcal{A}_j} \neq (x \wedge y)^{\mathcal{B}}$. It is easy to see that for this to be the case we must have that $x$ and $y$ are lines which are parallel in $\mathcal{A}_j$ but incident in $\mathcal{B}$. Let $t \in J$ be such that $(x \wedge y)^{\mathcal{B}} = a_t$. Notice that $t \geq j$ because $\mathcal{A}_j = \langle A, (a_i)_{i \in J, \, i < j} \rangle_{\mathcal{B}}$. Of course, $x, y \in \langle A, (a_i)_{i \in J, \, i < j} \rangle_{\mathcal{C}} \subseteq \langle A, (a_i)_{i < j} \rangle_{\mathcal{C}} \subseteq \langle A, (a_i)_{i < t} \rangle_{\mathcal{C}}$, and by hypothesis $(x \wedge y)^{\mathcal{B}} = (x \wedge y)^{\mathcal{C}} =a_t$, because $\mathcal{B} \preccurlyeq \mathcal{C}$ implies $\mathcal{B} \leq \mathcal{C}$. Thus, $\langle A, (a_i)_{i < t} \rangle_{\mathcal{C}}$ is not a $\wedge$-subgeometry of $\mathcal{C}$, which contradicts the fact that $\mathcal{A} \preccurlyeq \mathcal{C}$.
We now prove (6). We show that $\mathrm{LS}(\mathbf{K}^3, \preccurlyeq) = \omega$. Let $\mathcal{A} \in \mathbf{K}^3$ and $B \subseteq A$, we want to find $\mathcal{C} \in \mathbf{K}^3$ such that $B \subseteq C$, $\mathcal{C} \preccurlyeq \mathcal{A}$ and $|C| \leq |B| + \omega$. Let $\mathcal{A} = \mathrm{CS}(FG_3, (a_i, B_i)_{i < \alpha})$ and $X \subseteq \alpha$ such that $|X| \leq |B| + \omega$ and $B \subseteq \langle FG_3, (a_i)_{i \in X} \rangle_{\mathcal{A}}$ -- such an X can be easily found, just add enough points to lines in $B$ so that every line is the sup of some points. By Lemma \ref{Konig}, we can assume that $X$ is closed and $|X| \leq |B| + \omega$. By Lemma \ref{closed_implies_strong_subgeo}, we have that
	\[ \langle FG_3, (a_i)_{i \in X} \rangle_{\mathcal{A}} = \mathrm{CS}(FG_3, (a_i, B_i)_{i \in X}) \preccurlyeq \mathcal{A},\]
and so we are done. 
\end{proof}

\begin{theorem}\label{th_fail_smooth} The Smoothness Axiom fails for $(\mathbf{K}^{3}, \preccurlyeq)$, and $(\mathbf{K}^{3}, \preccurlyeq) \neq (\mathbf{K}^{3}_{*}, \preccurlyeq^*)$.
\end{theorem}

\begin{proof} Let $\mathcal{B}_2$ be the plane whose points and non-trivial lines, i.e. lines with at least three points, are represented in Figure \ref{fail_smoth}. Then, following the well-order as in the figure, i.e. $D_0 < E_0 < H_0 < D_1 < C_0 < B_0 < E_1 < \cdots < B_2$, it is easy to see that $\mathcal{B}_2 \in \mathbf{K}^{3}$. Let now $\mathcal{A}_2 = \langle A, B, C, (C_i, B_i)_{i \leq 2} \rangle_{\mathcal{B}_2}$. Then, using e.g. the induced ordering, i.e. $C_0 < B_0 < C_1 < B_1 < C_2 < B_2$, one sees that also $\mathcal{A}_2 \in \mathbf{K}^{3}$. Furthermore, the induced reverse ordering, i.e. $D_3 < H_2 < E_2 < D_2 < H_1 < \cdots < D_0$, is a witness for $\mathcal{A}_2 \preccurlyeq \mathcal{B}_2$. Finally, no linear extension of $D_0 < D_1$ or $D_1 < D_2$ or $D_2 < D_3$ or $D_0 < D_2$ or $D_0 < D_3$ or $D_1 < D_3$ is a witness for $\mathcal{A}_2 \preccurlyeq \mathcal{B}_2$. I.e. any witness for $\mathcal{A}_2 \preccurlyeq \mathcal{B}_2$ has to be a linear extension of $D_3 < D_2 < D_1 < D_0$, as in the exhibited witness.

\smallskip
\noindent
Using the same construction, i.e. keep going as in $B_2 < E_3 < H_3 < D_4 < C_3 < B_3 < \cdots$, we can build a $\preccurlyeq$-chain $(\mathcal{B}_i)_{i < \omega}$ and then consider for every $i < \omega$ the structure $\mathcal{A}_i = \langle A, B, C, (C_j, B_j)_{j \leq i} \rangle_{\mathcal{B}_i}$. Let now
$$ \mathcal{A}_{\omega} = \bigcup_{i < \omega} \mathcal{A}_i \;\; \text{ and } \;\; \mathcal{B}_{\omega} = \bigcup_{i < \omega} \mathcal{B}_i.$$
Then $(\mathcal{A}_i)_{i < \omega}$ is a $\preccurlyeq$-chain (easy to verify) and $\mathcal{A}_i \preccurlyeq \mathcal{B}_{\omega}$ for each $i < \omega$, because $\mathcal{A}_i \preccurlyeq \mathcal{B}_i$ via the induced reverse ordering (as already noticed for $i = 2$) and of course $\mathcal{B}_i \preccurlyeq \mathcal{B}_{\omega}$. On the other hand $\mathcal{A}_{\omega} \not\preccurlyeq \mathcal{B}_{\omega}$, because our constructions are well-founded and to construct $\mathcal{B}_{\omega}$ from $\mathcal{A}_{\omega}$ we would be in need of a linear extension of the linear order $D_0 > D_1 > D_2 > D_3 > D_4 > \cdots$. As a matter of facts, the linear order $\cdots < D_n < H_{n-1} < E_{n-1} < D_{n-1} < \cdots < D_3 < H_2 < E_2 < D_2 < \cdots < D_0$ is a witness for $\mathcal{A}_{\omega} \preccurlyeq^* \mathcal{B}_{\omega}$. This shows that $(\mathbf{K}^{3}, \preccurlyeq) \neq (\mathbf{K}^{3}_{*}, \preccurlyeq^*)$.
\begin{figure}[ht]
	\begin{center}
	\begin{tikzpicture}
\matrix (a) [matrix of math nodes, column sep=0.4cm, row sep=0.3cm]{
  &   &   &   D_0E_0D_1 & E_0H_0C_0 & H_0D_1B_0 & \\
A & B & C & D_0 & E_0 	    & H_0       & D_1 & C_0 & B_0 \\
  &   &   &   D_1E_1D_2 & E_1H_1C_1 & H_1D_2B_1 & \\
  &  D_1 &  C_0 & B_0 & E_1 	    & H_1       & D_2 & C_1 & B_1 \\
  &   &   &   D_2E_2D_3 & E_2H_2C_2 & H_2D_3B_2 & \\
  &  D_2  & C_1 & B_1 & E_2 	    & H_2       & D_3 & C_2 & B_2 \\};
\end{tikzpicture}
\end{center} \caption{Failure of Smoothness Axiom.} \label{fail_smoth} \end{figure}
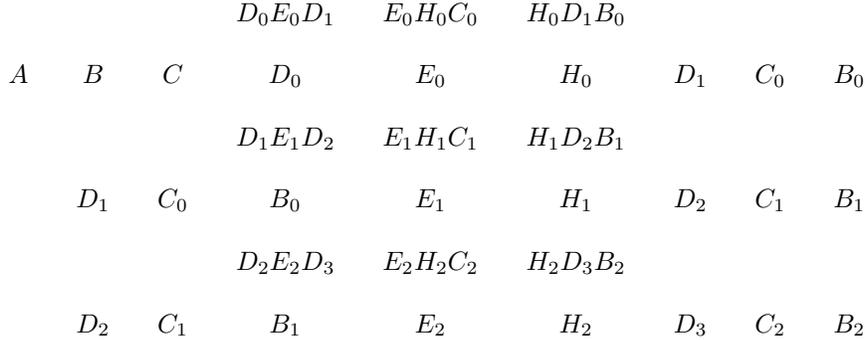
\end{proof}

\begin{remark} Notice that in the above construction if we let for $i < \omega$ 
$$\mathcal{C}_i = \langle A_{\omega}, B_{\omega} - B_i \rangle_{\mathcal{B}_{\omega}},$$
then for every $i, j < \omega$ we have that $\mathcal{A}_{j} \preccurlyeq \mathcal{C}_i \preccurlyeq \mathcal{B}_{\omega}$, and $\bigcap_{i < \omega} \mathcal{C}_i = \mathcal{A}_{\omega}$. 
\end{remark}

\begin{remark} It is possible to extend the structures above to structures $(\mathcal{B}^{*}_i)_{i < \omega}$, $(\mathcal{A}^{*}_i)_{i < \omega}$, $\mathcal{B}^{*}_{\omega}$ and $\mathcal{A}^{*}_{\omega}$, so that in addition they are {\em full}, i.e. saturated with respect to Galois types -- just extend any of the structures with a limit construction that adds cofinally many generic point and cofinally many points in each (newly created) line. It is then straightforward to verify that also in this case we have that $(\mathcal{A}^{*}_i)_{i < \omega}$ is a $\preccurlyeq$-chain, $\mathcal{A}^{*}_i \preccurlyeq \mathcal{B}^{*}_{\omega}$ for each $i < \omega$, but $\mathcal{A}^{*}_{\omega} \not\preccurlyeq \mathcal{B}^{*}_{\omega}$. This shows that the Smoothness Axiom fails even if we strengthen $\preccurlyeq$ with (fragments of) the logic $L_{\infty, \omega}$, and so it is somehow an inherent property of $(\mathbf{K}^3, \preccurlyeq)$.
\end{remark}

	\begin{example} The following counterexample shows that the strategy used in Theorem \ref{our_class_is_almost_AEC} to prove the coherence of $\preccurlyeq$ does not work in rank $n \geq 4$. Let $\mathcal{C}$ be the structure obtained from $FG_4$ on atoms $A, B, C, D$ by adding $E$ and $F$ generic, $G$ and $H$ under the hyperplane $AEF$, $I$ and $L$ under the hyperplane $BEF$, and $M$ and $N$ under the line $EF$. Then, of course, $\mathcal{C} \in \mathbf{K}^4$. In Figure \ref{fail_strat} we see a partial representation of $\mathcal{C}$, only the points and relevant non-trivial lines and hyperplanes of $\mathcal{C}$ have been represented. Let now 
	\[ \mathcal{A} = \langle A, B, C, D \rangle_{\mathcal{C}} = FG_4 \;\; \text{ and }  \;\; \mathcal{B} = \langle A, B, C, D, G, H, I, L, M, N \rangle_{\mathcal{C}}. \]
Then, of course, $\mathcal{A} \preccurlyeq \mathcal{C}$ and $\mathcal{A} \leq \mathcal{B}$. Furthermore, we can construct $\mathcal{C}$ from $\mathcal{B}$, by adding $E$ and $F$ under the line $MN$. But $\mathcal{B}$ can not be built from $\mathcal{A}$ following the order induced by the order used to build $\mathcal{C}$, i.e. $G < H< I < L < M < N$. As in fact, at the last stage, $N$ should be added to both hyperplanes $AEFGHMN$ and $BEFILMN$ because they are already existing in $\langle A, B, C, D, G, H, I, L, M \rangle_{\mathcal{C}}$.
	
\begin{figure}[ht]
\begin{center}
\begin{tikzpicture}
\matrix (a) [matrix of math nodes, column sep=0.16cm, row sep=0.3cm]{
  &   &   & AEFGHMN & & & BEFILMN \\
  &   &   &   &   & EFMN  \\
A & B & C & D & E & F & G & H & I & L & M & N \\};
\end{tikzpicture}
\end{center} \caption{Failure of strategy for coherence.} \label{fail_strat} \end{figure}
\end{example}

	\begin{oproblem} Does $(\mathbf{K}^{n}, \preccurlyeq)$ satisfies the Coherence Axiom for $n \geq 4$?		
\end{oproblem}


	
	We now define a notion of independence in $(\mathbf{K}^{3}, \preccurlyeq)$. 
	
	\begin{definition} Let $\mathcal{A} \preccurlyeq \mathcal{B} \preccurlyeq \mathfrak{M}$ and $a \in \mathfrak{M}^{< \omega}$.
Then, we define  $a \pureindep[\mathcal{A}] \mathcal{B}$ if there exists $\mathcal{A} \preccurlyeq \mathcal{C}_a$ such that $a \in C^{< \omega}_a$ and $\langle C_a, B \rangle_{\mathfrak{M}} = \mathcal{C}_a \oplus_{\mathcal{A}} \mathcal{B} \preccurlyeq \mathfrak{M}$. 
\end{definition}
Notice that from the above definition it follows that $\mathcal{C}_a \preccurlyeq \mathfrak{M}$, furthermore $\mathcal{C}_a$ can always be taken so that $C_a - A$ is finite.

	\begin{example} We now give two concrete examples of lack of independence. Let $\mathcal{H} \preccurlyeq \mathfrak{M}$ be the plane in Figure \ref{lack_indep}, and $\mathcal{A}_0 = \langle A, B, C, E \rangle_{\mathcal{H}}$ and $\mathcal{B}_0 = \langle A, B, C, E, F \rangle_{\mathcal{H}}$. Then $D \not\!\pureindep[\mathcal{A}_0] \mathcal{B}_0$, because any $\mathcal{A}_0 \preccurlyeq \mathcal{C}_D \preccurlyeq \mathfrak{M}$ with $D \in C_D$ and $C_D \cap B_0 = A_0$ is so that $\langle C_D, B_0 \rangle_{\mathfrak{M}} \neq \mathcal{C}_D \oplus_{\mathcal{A}_0} \mathcal{B}_0$, as in fact in the structure $\langle C_D, B_0 \rangle_{\mathfrak{M}}$ the set $\left\{ D, E, F \right\}$ is dependent, while in $\mathcal{C}_D \oplus_{\mathcal{A}_0} \mathcal{B}_0$ it is independent. Let now $\mathcal{A}_1 = \langle A, B, C \rangle_{\mathcal{H}}$ and $\mathcal{B}_1 = \langle A, B, C, F \rangle_{\mathcal{H}}$. Then $D \not\!\pureindep[\mathcal{A}_1] \mathcal{B}_1$, because any $\mathcal{A}_1 \preccurlyeq \mathcal{C}_D \preccurlyeq \mathfrak{M}$ with $D \in C_D$ and $C_D \cap B_1 = A_1$ is so that $\langle C_D, B_1 \rangle_{\mathfrak{M}} \not\preccurlyeq \mathfrak{M}$ as in fact $(EBC  \wedge DEF)^{\langle C_D, B_1 \rangle_{\mathfrak{M}}} = 0$, while $(EBC  \wedge DEF)^{\mathfrak{M}} = E$.
\begin{figure}[ht]
	\begin{center}
	\begin{tikzpicture}
\matrix (a) [matrix of math nodes, column sep=0.2cm, row sep=0.3cm]{
AB & AC  & BCE  & AD & BD & CD & AE & DEF & AF & BF & CF \\
A  & B   & C    &    & D   &    & E  &     & F \\};
\end{tikzpicture}
\end{center} \caption{Examples of lack of independence.} \label{lack_indep} \end{figure}
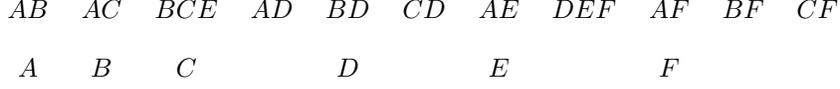
\end{example}
	
	\begin{theorem}\label{good_frame_th} $(\mathbf{K}^{3},  \preccurlyeq, \pureindep)$ is an independence calculus.
\end{theorem}

\begin{proof}  (1) and (2) have already been shown. We now come to (3). (a) is OK.
We prove (b). Let $\mathcal{A}_0 \preccurlyeq \mathcal{A}_1 \preccurlyeq \mathfrak{M}$, $a \in \mathcal{A}_1^{< \omega}$ and assume that $b \pureindep[\mathcal{A}_0] \mathcal{A}_1$, i.e. there exists $\mathcal{A}_0 \preccurlyeq \mathcal{C}_{b}$ such that $b \in C^{< \omega}_{b}$ and $\langle C_{b}, A_1 \rangle_{\mathfrak{M}} = \mathcal{C}_{b} \oplus_{\mathcal{A}_0} \mathcal{A}_1  \preccurlyeq \mathfrak{M}$. We want to show that $\langle A_1, C_{b} \rangle_{\mathfrak{M}} = \mathcal{A}_1 \oplus_{\mathcal{A}_0} \mathcal{C}_{b} \preccurlyeq \mathfrak{M}$, this of course establishes $a \pureindep[\mathcal{A}_0] \mathcal{C}_{b}$. But this is clear since by Proposition \ref{amalg_strong} we have that $\mathcal{C}_{b} \oplus_{\mathcal{A}_0} \mathcal{A}_1 = \mathcal{A}_1 \oplus_{\mathcal{A}_0} \mathcal{C}_{b}$.

\smallskip
\noindent
We prove (c). Let $\mathcal{A} \preccurlyeq \mathcal{A}' \preccurlyeq \mathcal{B}' \preccurlyeq \mathcal{B} \preccurlyeq \mathfrak{M}$, and suppose that $a \pureindep[\mathcal{A}] \mathcal{B}$, we want to show that $a \pureindep[\mathcal{A}'] \mathcal{B}'$. Let 
	\[ \mathrm{CS}(\mathcal{A}, (c_i, C_i)_{i < \alpha + \beta + \gamma + n}) \]
witness the chain of $\preccurlyeq$-embeddings
	\[ \mathcal{A} \preccurlyeq \mathcal{A}' \preccurlyeq \mathcal{B}' \preccurlyeq \mathcal{B} \oplus_{\mathcal{A}} \mathcal{C}_a = \mathcal{C}_a \oplus_{\mathcal{A}} \mathcal{B} = \langle C_a, B \rangle_{\mathfrak{M}}, \]
with $(c_i, C_i)_{\alpha + \beta + \gamma \leq i < \alpha + \beta + \gamma + n}$ witnessing that $\mathcal{A} \preccurlyeq \mathcal{C}_a$, as in Definition \ref{free_amalgamation}. We want to show that $\langle C_a,  B' \rangle_{\mathfrak{M}} = \mathcal{C}_a \oplus_{\mathcal{A}'} \mathcal{B}'$, this of course establishes $a \pureindep[\mathcal{A}'] \mathcal{B}'$ since $\mathcal{C}_a \oplus_{\mathcal{A}'} \mathcal{B}' \preccurlyeq \mathcal{C}_a \oplus_{\mathcal{A}} \mathcal{B} \preccurlyeq \mathfrak{M}$. To this extent, let 
	\[ \pi: \alpha + \beta + n + \gamma \rightarrow \alpha + \beta + \gamma + n\]
be the obvious bijection. Then $\pi$ is as in Lemma \ref{change_order}, and so 	
	\[ \mathrm{CS}(\mathcal{A}, (c_i, C_i)_{i < \alpha + \beta + \gamma + n}) = \mathrm{CS}(\mathcal{A}, (c_{\pi(i)}, C_{\pi(i)})_{i < \alpha + \beta + n + \gamma}). \]
Thus,
	\[\langle C_a, B' \rangle_{\mathfrak{M}} = \mathrm{CS}(\mathcal{A}, (c_{\pi(i)}, C_{\pi(i)})_{i < \alpha + \beta + n}) = \mathrm{CS}(\mathcal{A}', (c_{\pi(i)}, C_{\pi(i)})_{\alpha \leq i < \alpha + \beta + n}).\]
We prove (d). We prove more, i.e. for every $a \in \mathfrak{M}^{< \omega}$ and $\mathcal{B} \preccurlyeq \mathfrak{M}$ there exists finite $\mathcal{A} \preccurlyeq \mathcal{B}$ such that $a \pureindep[\mathcal{A}] \mathcal{B}$. Let $\mathcal{B} \preccurlyeq \mathcal{D} \preccurlyeq \mathfrak{M}$ be such that $a \in D^{< \omega}$, and let
	\[ \mathrm{CS}(FG_3, (a_i, B_i)_{i < \alpha + \beta}) \]
witness that $FG_3 \preccurlyeq \mathcal{B} \preccurlyeq \mathcal{D}$. Then we can find finite closed $X \subseteq \alpha + \beta$ such that $a \in (\langle FG_3, (a_i)_{i \in X} \rangle_{\mathfrak{M}})^{< \omega}$ by $\mathcal{B} \preccurlyeq \mathcal{D} \preccurlyeq \mathfrak{M}$, Lemma \ref{Konig} and Proposition \ref{closed_implies_subgeo}. Let now
	\[\mathcal{A} = \mathrm{CS}(FG_3, (a_i, B_i)_{i \in X \cap \alpha}) \;\; \text{ and } \;\; \mathcal{C}_a = \mathrm{CS}(FG_3, (a_i, B_i)_{i \in X}).\] 
Then, $\langle C_a, B \rangle_{\mathfrak{M}} = \mathcal{C}_{a} \oplus_{\mathcal{A}} \mathcal{B} \preccurlyeq \mathcal{D} \preccurlyeq \mathfrak{M}$.

\smallskip
\noindent
We prove (e). Let $\mathcal{A} \preccurlyeq \mathcal{B} \preccurlyeq \mathfrak{M}$, $a \in \mathfrak{M}^{< \omega}$ and $\mathcal{A} \preccurlyeq \mathcal{C}_a \preccurlyeq \mathfrak{M}$ such that $a \in C_a^{< \omega}$. Let now $f:\mathcal{C}_a \cong \mathcal{C}'_{a}$ be such that ${C}'_{a} \cap B = A$ and $f \restriction A = id_A$, then $\mathcal{A} \preccurlyeq \mathcal{C}'_{a}$ and so we can consider the structure $\mathcal{C}'_{a} \oplus_{\mathcal{A}} \mathcal{B}$. Now, $\mathcal{B} \preccurlyeq \mathfrak{M}$ and $\mathcal{B} \preccurlyeq \mathcal{C}'_{a} \oplus_{\mathcal{A}} \mathcal{B}$. Thus, by properties of our monster model, we can find a $\preccurlyeq$-embedding $g: \mathcal{C}'_{a} \oplus_{\mathcal{A}} \mathcal{B} \rightarrow \mathfrak{M}$ such that $g \restriction B = id_B$. Then $g \circ f: \mathcal{C}_a \cong g(\mathcal{C}'_a)$, $\mathcal{C}_a \preccurlyeq \mathfrak{M}$ and $g(\mathcal{C}'_a) \preccurlyeq \mathfrak{M}$, so by model homogeneity we can find $h \in \mathrm{Aut}(\mathfrak{M})$ extending $g \circ f$ and thus fixing $A$ pointwise, given that $g \circ f \restriction A = id_A$. Hence, $b = h(a)$ is so that $\mathrm{tp}(b/\mathcal{A}) = \mathrm{tp}(a/\mathcal{A})$ and $b \pureindep[\mathcal{A}] \mathcal{B}$.

\smallskip
\noindent
We prove (f). Let $\mathcal{A} \preccurlyeq \mathcal{B} \preccurlyeq \mathfrak{M}$, $a \pureindep[\mathcal{A}] \mathcal{B}$, $b \pureindep[\mathcal{A}] \mathcal{B}$ and suppose that $\mathrm{tp}(a/\mathcal{A}) = \mathrm{tp}(b/\mathcal{A})$. Let $\mathcal{A} \preccurlyeq \mathcal{C}_a$ be such that $a \in C^{< \omega}_a$ and $\langle C_a, B \rangle_{\mathfrak{M}} = \mathcal{C}_a \oplus_{\mathcal{A}} \mathcal{B} \preccurlyeq \mathfrak{M}$, and $\mathcal{A} \preccurlyeq \mathcal{C}_b$ such that $b \in C^{< \omega}_b$ and $\langle C_b, B \rangle_{\mathfrak{M}} = \mathcal{C}_b \oplus_{\mathcal{A}} \mathcal{B} \preccurlyeq \mathfrak{M}$.
By assumption $\mathrm{tp}(a/\mathcal{A}) = \mathrm{tp}(b/\mathcal{A})$, so there exists $f \in \mathrm{Aut}(\mathfrak{M}/\mathcal{A})$ such that $f(b) = a$. Let $\mathcal{D} \preccurlyeq \mathfrak{M}$ be such that $\mathcal{C}_a \preccurlyeq \mathcal{D}$ and $f(\mathcal{C}_b) \preccurlyeq \mathcal{D}$. 
%
Using the same argument used in the proof of (e), we can find $g \in \mathrm{Aut}(\mathfrak{M}/\mathcal{A})$ such that $\langle g(\mathcal{D}), B \rangle_{\mathfrak{M}} = g(\mathcal{D}) \oplus_{\mathcal{A}} \mathcal{B} \preccurlyeq \mathfrak{M}$. Let now $\mathcal{C}'_a = g(\mathcal{C}_a)$, $\mathcal{C}'_b = g \circ f (\mathcal{C}_b)$ and $\mathcal{D}' = g(\mathcal{D})$, then $\mathcal{C}'_a, \mathcal{C}'_b \preccurlyeq \mathcal{D}' \preccurlyeq \mathfrak{M}$. Thus, by properties of the free amalgam, we have that \begin{equation}\label{star}
\langle \mathcal{C}'_a, B \rangle_{\mathfrak{M}} = \mathcal{C}'_a \oplus_{\mathcal{A}} \mathcal{B} \preccurlyeq \mathfrak{M} \; \text{ and } \; \langle \mathcal{C}'_b, B \rangle_{\mathfrak{M}} = \mathcal{C}'_b \oplus_{\mathcal{A}} \mathcal{B} \preccurlyeq \mathfrak{M}.
\end{equation}
On the other hand, because of the independence assumption, we have that 
\begin{equation}\label{starstar}\langle \mathcal{C}_a, B \rangle_{\mathfrak{M}} = \mathcal{C}_a \oplus_{\mathcal{A}} \mathcal{B} \preccurlyeq \mathfrak{M} \; \text{ and } \; \langle \mathcal{C}_b, B \rangle_{\mathfrak{M}} = \mathcal{C}_b \oplus_{\mathcal{A}} \mathcal{B} \preccurlyeq \mathfrak{M}.
\end{equation}
Hence, we can find 
	$$\hat{t}: \langle \mathcal{C}_a, B \rangle_{\mathfrak{M}} = \mathcal{C}_a \oplus_{\mathcal{A}} \mathcal{B} \cong \mathcal{C}'_a \oplus_{\mathcal{A}} \mathcal{B} = \langle \mathcal{C}'_a, B \rangle_{\mathfrak{M}}$$
fixing $\mathcal{B}$ pointwise, and then use (1) and (2) and model homogeneity to extend $\hat{t}$ to a $t \in \mathrm{Aut}(\mathfrak{M}/\mathcal{B})$. Analogously, we can find 
	$$\hat{p}:  \langle \mathcal{C}'_b, B \rangle_{\mathfrak{M}} = \mathcal{C}'_b \oplus_{\mathcal{A}} \mathcal{B} \cong \mathcal{C}_b \oplus_{\mathcal{A}} \mathcal{B} = \langle \mathcal{C}_b, B \rangle_{\mathfrak{M}}$$
fixing $\mathcal{B}$ pointwise, and then use (1) and (2) and model homogeneity to extend $\hat{p}$ to a $p \in \mathrm{Aut}(\mathfrak{M}/\mathcal{B})$. But now we are done, as in fact $p \circ t \in \mathrm{Aut}(\mathfrak{M}/\mathcal{B})$ is a witness for $\mathrm{tp}(a/\mathcal{B}) = \mathrm{tp}(b/\mathcal{B})$, given that
	$$ a \xmapsto{t} g(a) = g \circ f(b) \xmapsto{p} b.$$

\smallskip
\noindent
Finally, we prove (g). Let $\delta$ be limit, $(\mathcal{A}_i)_{i \leq \delta}$ an increasing continuous $\preccurlyeq$-chain, $\mathcal{A}_{\delta} \preccurlyeq \mathfrak{M}$ and $a$, $(a_i)_{i < \delta}$ such that $a_i \pureindep[\mathcal{A}_0] \mathcal{A}_i$ and $\mathrm{tp}(a_i/\mathcal{A}_i) = \mathrm{tp}(a/\mathcal{A}_i)$ for every $i < \delta$. By (d), there exists $\alpha < \delta$ such that $a \pureindep[\mathcal{A}_{\alpha}] \mathcal{A}_{\delta}$. Furthermore, by (a), we have that $a \pureindep[\mathcal{A}_{0}] \mathcal{A}_{\alpha}$. Thus, by Proposition \ref{transitivity}, we have that $a \pureindep[\mathcal{A}_{0}] \mathcal{A}_{\delta}$, as wanted.

\end{proof} 

	\begin{corollary}\label{stability} $(\mathbf{K}^3, \preccurlyeq)$ is stable in every infinite cardinality.
\end{corollary}

\begin{proof} Let $a, b \in \mathfrak{M}^{< \omega}$ and $\mathcal{B} \preccurlyeq \mathfrak{M}$. By what we showed in (d) of the proof of Theorem \ref{good_frame_th}, there are finite $\mathcal{A}(a, \mathcal{B}), \mathcal{A}(b, \mathcal{B}) \preccurlyeq \mathcal{B}$ such that $a \pureindep[\mathcal{A}(a, \mathcal{B})] \mathcal{B}$ and $b \pureindep[\mathcal{A}(b, \mathcal{B})] \mathcal{B}$. By (f) of Theorem \ref{good_frame_th}, we then have that if $\mathcal{A}(a, \mathcal{B}) = \mathcal{A}(b, \mathcal{B})$ and $\mathrm{tp}(a/\mathcal{A}(a, \mathcal{B})) = \mathrm{tp}(b/ \mathcal{A}(b, \mathcal{B}))$, then $\mathrm{tp}(a/\mathcal{B}) = \mathrm{tp}(b/\mathcal{B})$. This suffices to establish stability in every infinite cardinality as $(\mathbf{K}^3, \preccurlyeq)$ is $\omega$-stable.
\end{proof}

\section{Restoring Smoothness}

	As announced, we come back to the study of $(\mathbf{K}_{*}^{n}, \preccurlyeq^*)$ and the associated pair $(\mathbf{K}_{+}^{n}, \preccurlyeq^+)$, which we will define later. We work under the assumption that $n = 3$. Our use of the assumption $n = 3$ in the development of the theory of $(\mathbf{K}_{*}^{3}, \preccurlyeq^*)$ and $(\mathbf{K}_{+}^{3}, \preccurlyeq^+)$ is limited to Lemmas \ref{change_order_spec} and \ref{canonical_determines_str}, and Proposition \ref{free_amalgamation_spec}. The proofs of these facts are rather technical, and require careful case distinctions, this is the main reason for assuming that $n = 3$. We conjecture that these facts can be proved for every $n < \omega$, and consequently that the theory of $(\mathbf{K}_{*}^{3}, \preccurlyeq^*)$ and $(\mathbf{K}_{+}^{3}, \preccurlyeq^+)$ can be lifted from planes to geometries of arbitrary finite rank. 
Our aim is to develop some of the theory developed for $(\mathbf{K}^{n}, \preccurlyeq)$ in the case of $(\mathbf{K}_{*}^{3}, \preccurlyeq^*)$ and $(\mathbf{K}_{+}^{3}, \preccurlyeq^+)$, so as to have for $(\mathbf{K}_{+}^{3}, \preccurlyeq^+)$ (which depends on $(\mathbf{K}_{*}^{3}, \preccurlyeq^*))$ results similar to the ones for $(\mathbf{K}^{n}, \preccurlyeq)$, while in addition restoring the Smoothness Axiom, which fails for $(\mathbf{K}^{n}, \preccurlyeq)$, as we saw in Theorem \ref{th_fail_smooth}. 

First of all, we introduce the notion of {\em canonical witness} for $\mathcal{A} \preccurlyeq^* \mathcal{B}$. This notion will be most useful in what follows. Given a plane $\mathcal{A}$ and a line $a$ of $\mathcal{A}$, we denote by $P(a) = P_{\mathcal{A}}(a)$ the set of points of $\mathcal{A}$ incident with the line $a$.

	\begin{definition} We say that $(a_i, B_i)_{i \in (I, <)}$ is a {\em canonical} witness for $\mathcal{A} \preccurlyeq^* \mathcal{B}$ if the following conditions are satisfied:
	\begin{enumerate}[i)]
	\item $(a_i, B_i)_{i \in (I, <)}$ is a witness for $\mathcal{A} \preccurlyeq^* \mathcal{B}$;
	\item if $|B_i| = 3$, then $B_i \subseteq P(\mathcal{A})$ and it is an independent set;
	\item if $|B_i| = 2$ and $\bigvee B_i \in A$, then for every $a_{j}, a_{k} \in P(\bigvee B_i)$ we have that $B_{j} = B_{k} = \left\{ a, b \right\}$, for fixed $\left\{ a, b \right\} \subseteq P(\mathcal{A})$;
	\item if $|B_i| = 2$ and $P(\bigvee B_i) \cap A = \left\{ a \right\}$, then $B_i = \left\{ a, a_{j} \right\}$ for $j$ least so that $a \vee a_{j} = \bigvee B_i$;
	\item if $|B_i| = 2$ and $P(\bigvee B_i) \cap A = \emptyset$, then $B_i = \left\{ a_{j}, a_{k} \right\}$ for $j$ and $k$ least so that $a_{j} \vee a_{k} = \bigvee B_i$.
\end{enumerate}
\end{definition}

	Clearly we can always find a canonical witness for $\mathcal{A} \preccurlyeq^* \mathcal{B}$.

	\begin{proposition}\label{canonical_determines_str} If $(a_i, B_i)_{i \in (I, <)}$ is a canonical witness for both $\mathcal{A} \preccurlyeq^* \mathcal{B}$ and $\mathcal{A} \preccurlyeq^* \mathcal{B}'$, then $\mathcal{B} = \mathcal{B}'$. 
\end{proposition}

	\begin{proof} It suffices to show that for every distinct $a, b, c \in P(\mathcal{B})$ with $|\left\{ a, b, c \right\} \cap A| \leq 2$ we have that
	$$ r(a \vee b \vee c) = 2 \, \text{ in } \, \mathcal{B} \;\; \Rightarrow \;\; r(a \vee b \vee c) = 2 \, \text{ in } \, \mathcal{B}'.$$
\newline {\bf Case A.} $\left\{ a, b, c \right\} - A = \left\{ a_i \right\}$. If $r(a \vee b \vee c) = 2$ in $\mathcal{B}$, $\mathcal{B} \models a \vee b \vee c = a \vee b$, then the set $\left\{ a, b, c \right\}$ is dependent in $\mathcal{B}$ and so $a_i$ is such that $B_{i} = \left\{ d, e \right\} \subseteq P(\mathcal{A})$ with $d \vee e = \bigvee (\left\{ a, b, c \right\} \cap A)$. Thus, $\mathcal{B}' \models a \vee b \vee c = a \vee b$, because also in $\mathcal{B}'$ the point $a_i$ gets added under the line $d \vee e = \bigvee (\left\{ a, b, c \right\} \cap A)$ which is in $\mathcal{A}$. 
\newline {\bf Case B.} $\left\{ a, b, c \right\} - A = \left\{ a_i, a_j \right\}$, for $i < j$. Let $\left\{ a, b, c \right\} \cap A = \left\{ x \right\}$ and suppose that $\mathcal{B} \models a \vee b \vee c = a \vee b$. Then either: i) $B_i = B_j = \left\{ d, e \right\} \subseteq P(\mathcal{A})$ with $d \vee e = a \vee b$, or ii) $B_j = \left\{ x, a_i \right\}$, or iii) $B_j = B_i = \left\{ x, a_t \right\}$ for $t < i$. In all these cases we are fine.
\newline {\bf Case C.} $\left\{ a, b, c \right\} - A = \left\{ a_{i}, a_{j}, a_{t} \right\}$, for $i < j < t$. Suppose that $\mathcal{B} \models a \vee b \vee c = a \vee b$.
\newline {\bf Case C.1.} $B_{t} \subseteq A$. Then $B_{i} = B_{j} = B_{t} = \left\{ d, e \right\} \subseteq P(\mathcal{A})$, and so $\mathcal{B}' \models a_{i} \vee a_{j} \vee a_{t} = d \vee e = a \vee b$.
\newline {\bf Case C.2.} $B_{t} \cap A = \left\{ a \right\}$. Then $B_{j} = B_{t} = \left\{ a, a_m \right\}$, for $m \leq i$. If $m < i$, we have $B_i = B_{j} = B_{t} = \left\{ a, a_m \right\}$. If $m = i$, we have $B_{j} = B_{t} = \left\{ a, a_i \right\}$. Hence, in both cases we have $\mathcal{B}' \models a_{i} \vee a_{j} \vee a_{t} = a \vee a_m = a \vee b$.
\newline {\bf Case C.3.} $B_{t} \cap A = \emptyset$. Then $B_{t} = \left\{ a_m, a_k \right\}$, for $m < k$, $m \leq i$ and $k \leq j$. 
\newline {\bf Case C.3.1.} $m = i$ and $k = j$. Then of course $\mathcal{B}' \models a_{i} \vee a_{j} \vee a_{t} = a \vee b$. 
\newline {\bf Case C.3.2.} $m < i$ and $k = j$. Then $\mathcal{B} \models a_{t} < a_{i} \vee a_{k}$ and $\mathcal{B} \models  a_t < a_m \vee a_k$. Thus, $\mathcal{B} \models a_{m} \vee a_{i} = a_m \vee a_k$ and so $\mathcal{B} \models  a_t < a_m \vee a_i$. Hence, $m$ and $k$ are not least, a contradiction.
\newline {\bf Case C.3.3.} $m = i$ and $k < j$. Then $\mathcal{B} \models a_{t} < a_{i} \vee a_{j}$ and $\mathcal{B} \models a_{t} < a_{i} \vee a_k$. Thus, $\mathcal{B} \models a_{i} \vee a_{j} = a_{i} \vee a_k$ and so $B_{j} = B_t = \left\{ a_m, a_k \right\}$. Hence, $\mathcal{B}' \models a_{i} \vee a_{j} \vee a_{t} = a_m \vee a_k = a_i \vee a_k = a \vee b$.
\newline {\bf Case C.3.4.} $m < i$ and $k < j$. Then $\mathcal{B} \models a_{t} < a_{i} \vee a_{j}$ and $\mathcal{B} \models  a_t < a_m \vee a_k$. Thus, $\mathcal{B} \models a_{m} \vee a_{i} = a_m \vee a_k$ and so $k \leq i$ otherwise, as $\mathcal{B} \models a_t < a_m \vee a_i$, $m$ and $k$ would not be least. Then $m < k \leq i < j$ and so $B_{j} = B_t = \left\{ a_m, a_k \right\}$. Hence, $\mathcal{B}' \models a_{i} \vee a_{j} \vee a_{t} = a_m \vee a_k = a \vee b$.
\end{proof}
	
	 We now prove an analog of Lemma \ref{change_order}.

\begin{lemma}\label{change_order_spec} Let $(a_i, B_i)_{i \in (I, <_0)}$ be a witness for $\mathcal{A} \preccurlyeq^* \mathcal{B}$ and $(I, <_1)$ a linear order such that 
\begin{enumerate}[i)]
\item for every $j \in I$, $B_{j} \subseteq P(\mathcal{A}) \cup \left\{ a_{i} \, | \, i <_1 j \right\}$;
\item for every $x \in B - A$  there exists $j \in I$ such that $x \in \langle A, (a_i)_{i \leq_1 j}\rangle_{\mathcal{B}} - \langle A, (a_i)_{i <_1 j}\rangle_{\mathcal{B}}$.
\end{enumerate} 
Then $(a_i, B_i)_{i \in (I, <_1)}$ is a witness for $\mathcal{A} \preccurlyeq^* \mathcal{B}$.
\end{lemma}

	\begin{proof} Suppose that the conclusion is false, then there exists $j \in I$ such that 
	\begin{equation}\label{star3}
 \langle A, (a_i)_{i \leq_1 j} \rangle_{\mathcal{B}} \neq \langle A, (a_i)_{i <_1 j} \rangle_{\mathcal{B}} \oplus_{\bigvee B_j} a_j. 
  \end{equation}
By assumption $B_{j} \subseteq P(\mathcal{A}) \cup \left\{ a_{i} \, | \, i <_1 j \right\}$, and so
	$$a_j \leq (\bigvee B_j)^{\langle A, (a_i)_{i \leq_1 j}\rangle_{\mathcal{B}}} = (\bigvee B_j)^{\mathcal{B}}.$$	
Thus, the only reason for (\ref{star3}) to happen is that there exists a line $l \neq \bigvee B_j$ in $\langle A, (a_i)_{i <_1 j}\rangle_{\mathcal{B}}$ such that $a_j \leq l, \bigvee B_j$. 
If $l \in A$, then $l, \bigvee B_j \in \langle A, (a_i)_{i <_0 j}\rangle_{\mathcal{B}}$, contradicting the fact that $(a_i, B_i)_{i \in (I, <_0)}$ is a witness for $\mathcal{A} \preccurlyeq^* \mathcal{B}$. Thus, $l \not\in A$ and so, by (ii), we can find  $t \in I$ such that $l \in \langle A, (a_i)_{i \leq_1 t}\rangle_{\mathcal{B}} - \langle A, (a_i)_{i <_1 t}\rangle_{\mathcal{B}}$, and of course $t <_1 j$, because $l \in\langle A, (a_i)_{i <_1 j}\rangle_{\mathcal{B}}$.
\newline {\bf Case A.} $l = a \vee a_t$ for $a \in P(\mathcal{A})$. Suppose that $t <_0 j$, then $l, \bigvee B_j \in \langle A, (a_i)_{i <_0 j}\rangle_{\mathcal{B}}$, contradicting the fact that $(a_i, B_i)_{i \in (I, <_0)}$ is a witness for $\mathcal{A} \preccurlyeq^* \mathcal{B}$. Thus $j <_0 t$, and so $\bigvee B_t = l = a \vee a_j$, but $$B_{t} \subseteq P(\mathcal{A}) \cup \left\{ a_{i} \, | \, i <_1 t \right\},$$ and so $l \in \langle A, (a_i)_{i <_1 t}\rangle_{\mathcal{B}}$, a contradiction.
\newline {\bf Case B.} $l = a_s \vee a_t$ for $s <_1 t <_1 j$. Suppose that $s, t <_0 j$, then $l, \bigvee B_j \in \langle A, (a_i)_{i <_0 j}\rangle_{\mathcal{B}}$, contradicting the fact that $(a_i, B_i)_{i \in (I, <_0)}$ is a witness for $\mathcal{A} \preccurlyeq^* \mathcal{B}$. Thus either $j <_0 t$ or $j <_0 s$.
\newline {\bf Case B.1.} $j <_0 t$.
\newline {\bf Case B.1.1.} $s <_0 t$. If this is the case, then $\bigvee B_t = l = a_j \vee a_s$, but $$B_{t} \subseteq P(\mathcal{A}) \cup \left\{ a_{i} \, | \, i <_1 t \right\},$$ and so $l \in \langle A, (a_i)_{i <_1 t}\rangle_{\mathcal{B}}$, a contradiction.
\newline {\bf Case B.1.2.} $t <_0 s$. If this is the case, then $\bigvee B_s = l = a_j \vee a_t$, but $$B_{s} \subseteq P(\mathcal{A}) \cup \left\{ a_{i} \, | \, i <_1 s <_1 t \right\},$$ and so $l \in \langle A, (a_i)_{i <_1 t}\rangle_{\mathcal{B}}$, a contradiction.
\newline {\bf Case B.2.} $j <_0 s$.
\newline {\bf Case B.2.1.} $s <_0 t$. Exactly as in Case B.1.1. 
\newline {\bf Case B.2.2.} $t <_0 s$. Exactly as in Case B.1.2.  
\end{proof}

	\begin{corollary}\label{change_order_for_canonical} Let $(a_i, B_i)_{i \in (I, <_0)}$ be a canonical witness for $\mathcal{A} \preccurlyeq^* \mathcal{B}$ and $(I, <_1~\!)$ a linear order such that for every $j \in I$, $B_{j} \subseteq P(\mathcal{A}) \cup \left\{ a_{i} \, | \, i <_1 j \right\}$.
Then $(a_i, B_i)_{i \in (I, <_1)}$ is a canonical witness for $\mathcal{A} \preccurlyeq^* \mathcal{B}$.
\end{corollary}

\begin{proof} We show that in this case condition ii) of Lemma \ref{change_order_spec} is automatically satisfied, and thus by that lemma $(a_i, B_i)_{i \in (I, <_1)}$ is a witness for $\mathcal{A} \preccurlyeq^* \mathcal{B}$. The canonicity of $(a_i, B_i)_{i \in (I, <_1)}$ is established with similar arguments. Let $x \in B - A$ be a line of $\mathcal{B}$ and $j \in I$ such that $x \in \langle A, (a_i)_{i \leq_0 j}\rangle_{\mathcal{B}} - \langle A, (a_i)_{i <_0 j}\rangle_{\mathcal{B}}$.
\newline {\bf Case a.} $x = a \vee a_j$, for $a \in P(\mathcal{A})$.
We want to show that $x \in \langle A, (a_i)_{i \leq_1 j} \rangle_{\mathcal{B}} - \langle A, (a_i)_{i <_1 j}\rangle_{\mathcal{B}}$. Of course $x \in \langle A, (a_i)_{i \leq_1 j} \rangle_{\mathcal{B}}$ because 
$$(a \vee a_j)^{\langle A, (a_i)_{i \leq_1 j} \rangle_{\mathcal{B}}} = (a \vee a_j)^{\mathcal{B}} = (a \vee a_j)^{\langle A, (a_i)_{i \leq_0 j} \rangle_{\mathcal{B}}}.$$ Suppose that $x \in \langle A, (a_i)_{i <_1 j}\rangle_{\mathcal{B}}$, then we can find $t <_1 j$ such that $x = a \vee a_t$. Of course $j <_0 t$, otherwise $x \in \langle A, (a_i)_{i <_0 j}\rangle_{\mathcal{B}}$. Thus, $B_t = \left\{ a, a_j \right\}$ because $(a_i, B_i)_{i \in (I, <_0)}$ is canonical. But then $B_t \not\subseteq P(\mathcal{A}) \cup \left\{ a_{i} \, | \, i <_1 t \right\}$, a contradiction.
\newline {\bf Case b.} $x = a_i \vee a_j$, for $i <_0 j$. Let $k = min\left\{ i, j \right\}$ and $m = max\left\{ i, j \right\}$ with respect to $<_1$. Arguing as above, one sees that $x \in \langle A, (a_i)_{i \leq_1 m} \rangle_{\mathcal{B}}$. Suppose that $x \in \langle A, (a_i)_{i <_1 m}\rangle_{\mathcal{B}}$, then we can find $t <_1 m$ such that $x = a_k \vee a_t$. Of course $j <_0 t$, otherwise $x \in \langle A, (a_i)_{i <_0 j}\rangle_{\mathcal{B}}$. Thus, $B_t = \left\{ a_i, a_j \right\} = \left\{ a_k, a_m \right\}$ because $(a_i, B_i)_{i \in (I, <_0)}$ is canonical. But then $B_t \not\subseteq P(\mathcal{A}) \cup \left\{ a_{i} \, | \, i <_1 t \right\}$, a contradiction.
\end{proof}

	We aim at proving an analog of Lemma \ref{closed_implies_strong_subgeo}.

\begin{definition} Let $(a_i, B_i)_{i \in (I, <)}$ be a witness for $\mathcal{A} \preccurlyeq^* \mathcal{B}$ and $X \subseteq I$. We say that $X$ is closed if for every $j \in X$, $B_{j} \subseteq P(\mathcal{A}) \cup \left\{ a_{i} \, | \, i < j \text{ and } i \in X \right\}$. 
\end{definition}

\begin{lemma}\label{Konig_spec} Let $(a_i, B_i)_{i \in (I, <)}$ be a witness for $\mathcal{A} \preccurlyeq^* \mathcal{B}$ and $X' \subseteq I$. Then there exists closed $X' \subseteq X \subseteq I$ such that $|X| \leq |X'| + \omega$.
\end{lemma}

\begin{proof} Similar to the proof of Lemma \ref{Konig}.
\end{proof}

	Let $(I, <)$ be a linear order and $X \subseteq I$. Let then $(X, <)$ and $(I - X, <)$ be the linear orders induced on $X$ and $I- X$ by $(I, <)$. We denote by $(I, <_X)$ the concatenation of the linear orders $(X, <)$ and $(I - X, <)$.

\begin{lemma}\label{closed_implies_strong_subgeo_spec} Let $(a_i, B_i)_{i \in (I, <)}$ be a canonical witness for $\mathcal{A} \preccurlyeq^* \mathcal{B}$ and $X \subseteq I$ closed. Then 
	$$\mathcal{A} \preccurlyeq^* \langle A, (a_{i})_{i \in (X, <) }\rangle_{\mathcal{B}} \preccurlyeq^* \mathcal{B}.$$
\end{lemma}

\begin{proof} Obviously $(I, <_X)$ fulfils the condition of Corollary \ref{change_order_for_canonical}, so $(a_i, B_i)_{i \in (I, <_X)}$ is a witness for $\mathcal{A} \preccurlyeq^* \mathcal{B}$, from which it follows that $\mathcal{A} \preccurlyeq^* \langle A, (a_{i})_{i \in (X, <) }\rangle_{\mathcal{B}} \preccurlyeq^* \mathcal{B}$.
\end{proof}

	We now define the {\em free $\preccurlyeq^*$-amalgam} of $\mathcal{A}$ and $\mathcal{B}$ over $\mathcal{C}$, in symbols $\mathcal{A} \oplus^*_{\mathcal{C}} \mathcal{B}$. Given a plane $\mathcal{A}$ and   $a \in P(\mathcal{A})$, we denote by $L(a) = L_{\mathcal{A}}(a)$ the set of lines of $\mathcal{A}$ incident with the point $a$, and with $L(\mathcal{A})$ the set of lines of $\mathcal{A}$.	
	
	\begin{definition/proposition}\label{free_amalgamation_spec} Let $\mathcal{A}, \mathcal{B}, \mathcal{C} \in \mathbf{K}_{*}^3$, with $\mathcal{C} \preccurlyeq^* \mathcal{A}, \mathcal{B}$ and $A \cap B = C$. Let $L^* = \left\{ a \vee b \, | \, a \in P(\mathcal{A}) - P(\mathcal{C}), b \in P(\mathcal{B}) - P(\mathcal{C}), L_{\mathcal{A}}(a) \cap L_{\mathcal{B}}(b) = \emptyset \right\}$, and define $\mathcal{A} \oplus_{\mathcal{C}}^{*} \mathcal{B} = \mathcal{D}$ as:
	\begin{enumerate}[i)]
	\item $P(\mathcal{D}) = P(\mathcal{A}) \cup P(\mathcal{B})$;
	\item $L(\mathcal{D}) = L(\mathcal{A}) \cup L(\mathcal{B}) \cup L^*$;
	\item $p \leq l$ iff $p, l \in A$ and $p \leq l$, or $p, l \in B$ and $p \leq l$, or $l = a \vee b \in L^*$ and $p = a$ or $p=b$.
	\end{enumerate}
	Then $\mathcal{A} \oplus_{\mathcal{C}}^{*} \mathcal{B}$ is a plane, $\mathcal{A}, \mathcal{B} \preccurlyeq^* \mathcal{A} \oplus_{\mathcal{C}}^{*} \mathcal{B}$ and $\mathcal{A} \oplus_{\mathcal{C}}^{*} \mathcal{B} \in \mathbf{K}_{*}^3$.
\end{definition/proposition}

	\begin{proof} Of course in $\mathcal{A} \oplus_{\mathcal{C}}^{*} \mathcal{B} = \mathcal{D}$ every element is a sup of points. To conclude that $\mathcal{D}$ is a plane it suffices to verify that for every $a, b \in D$ we have that $r(a \vee b) + r(a \wedge b) \leq r(a) + r(b)$, but this is easily seen with a standard case distinction. Furthermore, the linear order witnessing that $\mathcal{C} \preccurlyeq^* \mathcal{B}$ witnesses that $\mathcal{A} \preccurlyeq^* \mathcal{D}$, and the linear order witnessing that $\mathcal{C} \preccurlyeq^* \mathcal{A}$ witnesses that $\mathcal{B} \preccurlyeq^* \mathcal{D}$. Finally, $\mathcal{D} \in \mathbf{K}_{*}^3$ because there exists $\mathcal{E} \cong FG_3$ such that $\mathcal{E} \preccurlyeq^* \mathcal{C} \preccurlyeq^* \mathcal{A} \preccurlyeq^* \mathcal{D}$.
\end{proof}

	\begin{theorem}\label{*_AEC_but_cohe} $(\mathbf{K}_{*}^{3}, \preccurlyeq^*)$ satisfies all the $\mathrm{AEC}$ axioms except possibly the Coherence Axiom and the Smoothness Axiom. Furthermore it has $\mathrm{AP}$, $\mathrm{JEP}$ and $\mathrm{ALM}$.
\end{theorem}

\begin{proof} (1), (2), (3), (4.1) and (4.2) are OK. Regarding $\mathrm{ALM}$, $\mathrm{JEP}$ and $\mathrm{AP}$, the first is obviously satisfied, the second follows from the third because $FG_3$ $\preccurlyeq^*$-embeds in every structure in $\mathbf{K}^{3}_*$, and the third is taken care of by the free $\preccurlyeq^*$-amalgam $\mathcal{A} \oplus^*_{\mathcal{C}} \mathcal{B}$. Item (6) is as in Theorem \ref{our_class_is_almost_AEC}, where in this case we use Lemma \ref{Konig_spec} and Lemma \ref{closed_implies_strong_subgeo_spec}, of course. 
\end{proof}

	Unfortunately, we were not able to determine if $(\mathbf{K}_{*}^{3}, \preccurlyeq^*)$ satisfies the Coherence Axiom, and we were able to prove only a weak form of the Smoothness Axiom.
	
	\begin{lemma}\label{ultraproduct_lemma} Let $(\mathcal{A}_j)_{j < \delta}$ be a continuous increasing $\preccurlyeq^*$-chain in $\mathbf{K}_{*}^{3}$ and suppose that $\mathcal{A}_j \preccurlyeq^* \mathcal{B} \in \mathbf{K}_{*}^{3}$, then there exists $\mathcal{C} \in \mathbf{K}_{*}^{3}$ so that $\bigcup_{i < \delta} \mathcal{A}_i = \mathcal{A} \preccurlyeq^* \mathcal{C}$ and $\mathcal{B} \preccurlyeq^* \mathcal{C}$. 
\end{lemma}

	\begin{proof} First of all we can assume that $\delta = \kappa$ is a regular cardinal (otherwise replace $\delta$ with its own cofinality). Let $\mathcal{D} \cong FG_3$ be such that $\mathcal{D} \preccurlyeq^* \mathcal{A}_0 \preccurlyeq^* \mathcal{B}$ (such $\mathcal{D}$ exists as $\mathcal{A}_0 \in \mathbf{K}^3_*$ and $\preccurlyeq$ is transitive), and for $\alpha < \kappa$ let $(a_i, B^{\alpha}_i)_{i \in (J, <_{\alpha})}$ so that it satisfies the following conditions:
\begin{enumerate}[a)]
\item $a_i = i$, for every $i \in J$ (notational convenience);
\item $(a_i, B^{\alpha}_i)_{i \in (J, <_{\alpha})}$ is a canonical witness for $\mathcal{D} \preccurlyeq^* \mathcal{B}$;
\item $P(\mathcal{A}_{\alpha}) - P(\mathcal{D})$ is an initial segment of $(J, <_{\beta})$, for every $\alpha \leq \beta < \kappa$;
\item for every $a \in P(\mathcal{A})$, if $a \in P(\mathcal{A}_{\alpha})$ then $B^{\beta}_a = B^{\alpha}_a$ for every $\alpha \leq \beta$.
\end{enumerate} 
For every $\alpha < \kappa$, using the witness $(a_i, B^{\alpha}_i)_{i \in (J, <_{\alpha})}$ we expand the structure $\mathcal{B}$ to a structure $\mathcal{B}_{\alpha}$ by naming $\mathcal{D} - \left\{ 0, 1 \right\}$, adding a binary predicate $<$ interpreted as $<_{\alpha}$, and adding a binary predicate $R$ whose intended meaning is that $R(b, a)$ iff $b \in B^{\alpha}_a$. Let $U$ be an ultrafilter on $\kappa$ generated by end segments of $\kappa$, and consider the following ultraproducts
	$$ \mathcal{C}_0 = \prod_{\alpha < \kappa} \mathcal{B}_{\alpha}/U \; \text{ and } \; \mathcal{C}_1 = \prod_{\alpha < \kappa} \mathcal{B}_0/U.$$
Let $\mathcal{C}$ be the reduct of $\mathcal{C}_0$ (or equivalently $\mathcal{C}_1$) to the signature $L = \left\{ 0, 1, \vee, \wedge \right\}$ and $\eta$ the elementary embedding of $\mathcal{B}$ into $\mathcal{C}$ through constant functions. We show that:
\begin{enumerate}[i)]
	\item $\mathcal{C} \in \mathbf{K}_{*}^{3}$;
	\item $\eta(\mathcal{A}) \preccurlyeq^* \mathcal{C}$;
	\item $\eta(\mathcal{B}) \preccurlyeq^* \mathcal{C}$.
\end{enumerate}
We show i). First of all notice that $\mathcal{C}$ is a plane because, as already noticed, by Proposition \ref{equiv_def_geom_lat} the class of geometric lattices of fixed rank $n$ is first-order axiomatizable. Let $<^* = <^{\mathcal{C}_0}$, then $<^*$ is a linear order on $P(\mathcal{C}) - P(\eta(\mathcal{D})) = I$. For every $a \in I$, let $B^*_a = \left\{ b \in I \, | \, \mathcal{C}_0 \models R(b, a) \right\}$. Then $(a, B^*_a)_{a \in (I, <^*)}$ is a canonical witness for $\mathcal{D} \cong \eta(\mathcal{D}) \preccurlyeq^* \mathcal{C}$. This is because being a canonical witness is a first-order property (with respect to the expanded signature).

\smallskip
\noindent
We show ii). We show that $\eta(P(\mathcal{A}))$ is closed in $(a, B^*_a)_{a \in (I, <^*)}$, by Lemma \ref{closed_implies_strong_subgeo_spec} this suffices. Let then $a \in \eta(P(\mathcal{A}))$, we have to show that $B^*_{a} \subseteq \eta(P(\mathcal{D})) \cup \left\{ b \in I \, | \, b <^* a \right\}$. Let $\alpha < \kappa$ be such that $a \in \eta(P(\mathcal{A}_{\alpha}))$, then for all $\beta \geq \alpha$ we have that $B^{\beta}_a = B^{\alpha}_a$. Furthermore,  $|B^{\alpha}_a| \leq 3$. Thus, $B^*_{a} = \left\{ \eta(b) \in C \, | \, b \in B^{\alpha}_a \right\}$. Also, for all $\beta \geq \alpha$, $B^{\beta}_a \subseteq P(\mathcal{D}) \cup \left\{ b \in P(\mathcal{B}) \, | \, b <_{\beta} a \right\}$, and so $B^*_{a} \subseteq \eta(P(\mathcal{D})) \cup \left\{ b \in I \, | \, b <^* a \right\}$.

\smallskip
\noindent
Finally, iii) is as in ii), where now one uses the canonical witness $(I, B^{+}_a)_{a \in (I, <^{+})}$, for $<^{+} = <^{\mathcal{C}_1}$, and chooses the $B^{+}_a$ in analogy with the $B^*_a$ of the proof of ii). Then $\eta(P(\mathcal{B}))$ is closed in $(I, B^{+}_a)_{a \in (I, <^{+})}$, and so we can argue as in ii).
\end{proof}

	We digress momentarily from the study of $\preccurlyeq^*$ introducing a new relation $\preccurlyeq^{**}$. This relation, although similar to $\preccurlyeq^*$, is easily seen to be coherent.

\begin{definition}\label{def_of_**} 
\begin{enumerate}[i)]
	\item For $\mathcal{A}, \mathcal{B} \in \mathbf{K}^{3}_{0}$ with $\mathcal{A} \leq \mathcal{B}$, we let $\mathcal{A} \preccurlyeq^{**} \mathcal{B}$ if and only if there exists a linear ordering $(a_i)_{i \in (I, <)}$ of $P(\mathcal{B}) - P(\mathcal{A})$ such that for every $j \in I$ we have that $\langle A, (a_i)_{i < j} \rangle_{\mathcal{B}} \leq \mathcal{B}$, and if $J \subseteq I$ is an infinite descending sequence, then $r(\bigvee_{j \in J} a_j) = 3$.
	\item The class $\mathbf{K}^{3}_{**}$ is the class
$$\left\{ \mathcal{A} \in \mathbf{K}^{3}_{0} \, | \, \, \exists \, \mathcal{B} \cong FG_3 \text{ such that } \mathcal{B} \preccurlyeq^{**} \mathcal{A} \right\}.$$
\end{enumerate}
\end{definition}

	\begin{remark}\label{prop_of^**} It can be shown that $(\mathbf{K}_{**}^{3}, \preccurlyeq^{**})$ satisfies all the $\mathrm{AEC}$ axioms except possibly the Smoothness Axiom. Furthermore it has $\mathrm{AP}$, $\mathrm{JEP}$ and $\mathrm{ALM}$.
\end{remark}

	Unfortunately, we were not able to determine if $(\mathbf{K}_{**}^{3}, \preccurlyeq^{**})$ satisfies the Smoothness Axiom. The failure in determining the coherence of $\preccurlyeq^*$, combined with the failure in determining the smoothness of $\preccurlyeq^{**}$, led us to loosen up our relation $\preccurlyeq^*$ to a new relation $\preccurlyeq^+$, which is both coherent and smooth, as we will see. In accordance to our modification of $\preccurlyeq^*$ we move from $\mathbf{K}_{*}^{3}$ to a new class $\mathbf{K}_{+}^{3}$.
	
	\begin{definition}\label{def_of_+} 
\begin{enumerate}[i)]
\item The class $\mathbf{K}_{+}^{3}$ is the class
	$$\left\{ \mathcal{A} \in \mathbf{K}^{3}_{0} \, | \, \text{ there exists } \mathcal{B}, \mathcal{C} \in \mathbf{K}^{3}_{*} \text{ with } FG_3 \cong \mathcal{C} \leq \mathcal{A} \text{ and } \mathcal{C}, \mathcal{A} \preccurlyeq^* \mathcal{B} \right\}.$$
\item For $\mathcal{A}, \mathcal{B} \in \mathbf{K}_{+}^{3}$, we let $\mathcal{A} \preccurlyeq^+ \mathcal{B}$ iff $\mathcal{A} \leq \mathcal{B}$ and there exists $\mathcal{C} \in \mathbf{K}^{3}_{*}$ such that $\mathcal{A}, \mathcal{B} \preccurlyeq^* \mathcal{C}$.
\end{enumerate}
\end{definition}

	\begin{remark} Notice that if $(\mathbf{K}_{*}^{3}, \preccurlyeq^*)$ satisfies the Coherence Axiom, then we have that $$(\mathbf{K}_{*}^{3}, \preccurlyeq^*) = (\mathbf{K}_{+}^{3}, \preccurlyeq^+).$$ 
Notice also that $\mathcal{A} \preccurlyeq^* \mathcal{B}$ implies $\mathcal{A} \preccurlyeq^+ \mathcal{B}$.
\end{remark}
	
	\begin{theorem}\label{+_is_AEC} $(\mathbf{K}_{+}^{3}, \preccurlyeq^+)$ is an $\mathrm{AEC}$ with $\mathrm{AP}$, $\mathrm{JEP}$ and $\mathrm{ALM}$.
\end{theorem}

	\begin{proof} (1) and (2) (from Definition \ref{def_indep_first_order}) are OK. We prove (3). Of course we only prove the transitivity of $\preccurlyeq^+$. Suppose that $\mathcal{A} \preccurlyeq^+ \mathcal{B} \preccurlyeq^+ \mathcal{C}$, then $\mathcal{A} \leq \mathcal{B}$, $\mathcal{B} \leq \mathcal{C}$ and there are $\mathcal{C}', \mathcal{C}''\in \mathbf{K}_*^3$ such that $\mathcal{A}, \mathcal{B} \preccurlyeq^* \mathcal{C}'$, $\mathcal{B}, \mathcal{C} \preccurlyeq^* \mathcal{C}''$ and $C' \cap C'' = B$ (w.o.l.g.). Thus, $\mathcal{A} \leq \mathcal{C}$ and $\mathcal{A}, \mathcal{C} \preccurlyeq^* \mathcal{C}' \oplus^*_{\mathcal{B}} \mathcal{C}'' \in \mathbf{K}_*^3$.
	We prove (5). Let $\mathcal{A} \leq \mathcal{B} \preccurlyeq^+ \mathcal{C} \in \mathbf{K}_{+}^{3}$ and $\mathcal{A} \preccurlyeq^+ \mathcal{C}$. Then we can find $\mathcal{C}', \mathcal{C}'' \in \mathbf{K}_*^3$ such that $\mathcal{B}, \mathcal{C} \preccurlyeq^* \mathcal{C}'$, $\mathcal{A}, \mathcal{C} \preccurlyeq^* \mathcal{C}''$ and $C' \cap C'' = C$. Thus, $\mathcal{A} \leq \mathcal{B}$ and $\mathcal{A}, \mathcal{B} \preccurlyeq^* \mathcal{C}' \oplus^*_{\mathcal{C}} \mathcal{C}'' \in \mathbf{K}_*^3$, and so $A \preccurlyeq^+ B$, as wanted.
	Regarding $\mathrm{ALM}$, $\mathrm{JEP}$ and $\mathrm{AP}$, the first is obviously satisfied, the second follows from the third because $FG_3$ $\preccurlyeq^+$-embeds in every structure in $\mathbf{K}^{3}_+$, and the third is taken care of by the free $\preccurlyeq^*$-amalgam $\mathcal{A} \oplus^*_{\mathcal{C}} \mathcal{B}$. Specifically, let $\mathcal{C} \preccurlyeq^+ \mathcal{A}, \mathcal{B}$, then we can find $\mathcal{C}', \mathcal{C}'' \in \mathbf{K}^{3}_*$ such that $\mathcal{C}, \mathcal{A} \preccurlyeq^* \mathcal{C}'$, $\mathcal{C}, \mathcal{B} \preccurlyeq^* \mathcal{C}''$ and $C' \cap C'' = C$.  Thus, $\mathcal{A}, \mathcal{B} \preccurlyeq^* \mathcal{C}' \oplus^*_{\mathcal{C}} \mathcal{C}'' \in \mathbf{K}_*^3$.

\smallskip
\noindent	
	We prove (4). First we prove (4.3). Let $(\mathcal{A}_i)_{i < \delta}$ be an increasing continuous $\preccurlyeq^+$-chain in $\mathbf{K}_+^3$ with $\mathcal{A}_j \preccurlyeq^+ \mathcal{B} \in \mathbf{K}_+^3$ and assume that $\delta$ is limit. First of all, we construct $\mathcal{F} = \bigcup_{i < \delta} \mathcal{F}_i \in \mathbf{K}_*^3$ such that $\mathcal{A}_j, \mathcal{B} \preccurlyeq^* \mathcal{F}$. To this extent, for $i< \delta$, let $\mathcal{C}_i'' \in \mathbf{K}_*^3$ be such that $\mathcal{A}_i, \mathcal{B} \preccurlyeq^* \mathcal{C}_i''$ and so that if $i < j < \delta$, then $C_i'' \cap C_j'' = B$. For $i \leq \delta$, let then
	\begin{enumerate}[i)]
	\item $i = 0$: $\mathcal{F}_i = \mathcal{C}_0''$,
	\item $i = j+1$: $\mathcal{F}_i = \mathcal{F}_j \oplus^*_{\mathcal{B}} \mathcal{C}_i''$,
	\item $i$ limit: $\mathcal{F}_i = \bigcup_{j < i} \mathcal{F}_j$.
\end{enumerate}
We now define $(\mathcal{C}_i)_{i < \delta}$ and $(\mathcal{D}_i)_{i < \delta}$ $\preccurlyeq^*$-chains in $\in \mathbf{K}_*^3$ such that 
\begin{enumerate}[1)]
\item $FG_3 \cong \mathcal{E} \leq \mathcal{A}_0$;
\item $\mathcal{E} \preccurlyeq^* \mathcal{C}_0$;
\item $\mathcal{B} \preccurlyeq^* \mathcal{D}_0$;
\item for every $j \leq \delta$, $\bigcup_{i < j} \mathcal{A}_i \preccurlyeq^* \bigcup_{i < j} \mathcal{C}_i$;
\item for every $j < \delta$, $\mathcal{C}_j \preccurlyeq^* \bigcup_{i < \delta} \mathcal{D}_i$.
\end{enumerate}
Notice that the construction of $(\mathcal{C}_i)_{i < \delta}$ and $(\mathcal{D}_i)_{i < \delta}$ as above establishes (4.3), because then by Lemma \ref{ultraproduct_lemma} we can find $\mathcal{C} \in \mathbf{K}_*^3$ such that
$$\bigcup_{i < \delta} \mathcal{A}_i \preccurlyeq^* \bigcup_{i < \delta} \mathcal{C}_i \preccurlyeq^* \mathcal{C} \; \text{ and } \; \mathcal{B} \preccurlyeq^* \bigcup_{i < \delta} \mathcal{D}_i \preccurlyeq^* \mathcal{C}.$$
We first define $(\mathcal{D}_i)_{i < \delta}$. To this extent, let $FG_3 \cong \mathcal{E} \preccurlyeq^* \mathcal{C}_{-1}' \in \mathbf{K}^3_*$ be such that $\mathcal{E} \leq \mathcal{A}_0$ and $\mathcal{A}_0 \preccurlyeq^* \mathcal{C}_{-1}'$, and, for $i< \delta$, let $\mathcal{C}_i' \in \mathbf{K}_*^3$ be such that $\mathcal{A}_i, \mathcal{A}_{i+1} \preccurlyeq^* \mathcal{C}_i'$ and so that if $-1 \leq i < j < \delta$, then $C_i' \cap C_j' \cap F = A_{i+1}$. For $i \leq \delta$, let then
	\begin{enumerate}[a)]
	\item $i = 0$: $\mathcal{D}_i = \mathcal{F} \oplus^*_{\mathcal{A}_0} \mathcal{C}_{-1}'$;
	\item $i = j+1$: $\mathcal{D}_i = \mathcal{D}_j \oplus^*_{\mathcal{A}_i} \mathcal{C}_j'$;
	\item $i$ limit: $\mathcal{D}_i = \bigcup_{j < i} \mathcal{D}_j$.
\end{enumerate}
We then define $(\mathcal{C}_i)_{i < \delta}$ by letting for $i \leq \delta$
	\begin{enumerate}[a')]
	\item $i = 0$: $\mathcal{C}_i = \mathcal{C}_{-1}'$;
	\item $i = j+1$: $\mathcal{C}_i = \langle C_j, C_j' \rangle_{\mathcal{D}_i}$;
	\item $i$ limit: $\mathcal{C}_i = \bigcup_{j < i} \mathcal{C}_j$.
\end{enumerate}
Now, items 1), 2) and 3) are clearly satisfied. Also, for every $j \leq \delta$ we have that
$\bigcup_{i < j} \mathcal{A}_i \preccurlyeq^* \bigcup_{i < j} \mathcal{C}_i$,
as witnessed by the concatenation of the witnesses for $\mathcal{A}_0 \preccurlyeq^* \mathcal{C}_{-1}'$ and $\mathcal{A}_{i+1} \preccurlyeq^* \mathcal{C}_i'$, $0 < i < j$. Furthermore, for every $j < \delta$ we have that
	$$ \mathcal{C}_j \preccurlyeq^* \mathcal{F} \oplus^*_{\mathcal{A}_{j}} \mathcal{C}_j  = \mathcal{D}_{j} \preccurlyeq^* \bigcup_{i < \delta} \mathcal{D}_i.$$
To prove (4.1) and (4.2) just repeat the construction of $(\mathcal{C}_i)_{i < \delta}$ ignoring $(\mathcal{D}_i)_{i < \delta}$ (for $i = j+1$ let $\mathcal{C}_i = \mathcal{C}_j \oplus^*_{\mathcal{A}_i} \mathcal{C}_j'$). Notice that items 1), 2) and 4) of the list above establish (4.1), and item 4) establishes (4.2).

\smallskip
\noindent
	Finally, we prove that $\mathrm{LS}(\mathbf{K}_{+}^{3}, \preccurlyeq^+) = \omega$. Let $\mathcal{A} \in \mathbf{K}^3_+$ and $A_0 \subseteq A$, we want to find $\mathcal{A}^* \in \mathbf{K}^3_+$ such that $A_0 \subseteq A^*$, $\mathcal{A}^* \preccurlyeq^+ \mathcal{A}$ and $|A^*| \leq |A_0| + \omega$. Let $FG_3 \cong \mathcal{D} \leq \mathcal{A}$ and $\mathcal{B} \in \mathbf{K}^3_*$ be such that $\mathcal{D}, \mathcal{A} \preccurlyeq^* \mathcal{B}$. Without loss of generality we can assume that $P(\mathcal{D}) \subseteq A_0 \subseteq P(\mathcal{A})$. Let then $(a_i, B_i)_{i \in (I, <_0)}$ and $(b_i, C_i)_{i \in (J, <_1)}$ be canonical witnesses for $\mathcal{D} \preccurlyeq^* \mathcal{B}$ and $\mathcal{A} \preccurlyeq^* \mathcal{B}$, respectively. Let now $B_0 \subseteq P(\mathcal{B})$ be such that the following conditions are met:  
\begin{enumerate}[1)]
\item $A_0 \subseteq B_0$
\item if $a_i \in B_0$, then $B_i \subseteq B_0$; 
\item if $b_i \in B_0$, then $C_i \subseteq B_0$; 
\item $|B_0| \leq |A_0| + \omega$.
\end{enumerate} 
Let then $A^+ = A \cap B_0$, $\mathcal{A}^* = \langle A^+ \rangle_{\mathcal{B}}$ and $\mathcal{C} = \langle B_0 \rangle_{\mathcal{B}}$. We show that: i) $\mathcal{C} \in \mathbf{K}^3_*$, ii) $\mathcal{C} \preccurlyeq^* \mathcal{B}$ and iii) $\mathcal{A}^* \preccurlyeq^* \mathcal{C}$. This suffices, because then $\mathcal{A}^* \in \mathbf{K}^3_+$,
$\mathcal{A}^*, \mathcal{A} \preccurlyeq^* \mathcal{B} \in \mathbf{K}^3_*$
and of course $|A^*| \leq |B_0| + \omega \leq |A_0| + \omega$. The fact that $\mathcal{A}^* \in \mathbf{K}^3_+$ is because $\mathcal{D} \leq \mathcal{A}^*$ and $\mathcal{D}, \mathcal{A}^* \preccurlyeq^* \mathcal{C} \in \mathbf{K}^3_*$. The fact that $\mathcal{A}^*, \mathcal{A} \preccurlyeq^* \mathcal{B} \in \mathbf{K}^3_*$ is because $\mathcal{A}^* \preccurlyeq^* \mathcal{C} \preccurlyeq^* \mathcal{B}$ and by assumption $\mathcal{A} \preccurlyeq^* \mathcal{B}$. 
We then verify i) - iii). Items i) and ii) are easily established noticing that $X = \left\{ i \in I \, | \, a_i \in B_0 \right\}$ is closed in the sense of $(a_i, B_i)_{i \in (I, <_0)}$, and so by Lemma \ref{closed_implies_strong_subgeo_spec} we have what we want. Regarding iii), let  $J^* = \left\{ i \in J \, | \, b_i \in B_0 - A^+ \right\}$. Then, $J^*$ is closed in the sense of $(b_i, C_i)_{i \in (J, <_1)}$, and so by Lemma \ref{closed_implies_strong_subgeo_spec} we have that $(b_i, C_i)_{i \in (J^*, <_1)}$ witnesses that $\mathcal{A} \preccurlyeq^* \langle A, (b_i)_{i \in J^*} \rangle_{\mathcal{B}}$. Furthermore, $\bigcup_{i \in J^*} C_i \subseteq A^* \cup \left\{ b_i \, | \, i \in J^* \right\}$ and so $(b_i, C_i)_{i \in (J^*, <_1)}$ witnesses also that $\mathcal{A}^* \preccurlyeq^* \langle A^*, (b_i)_{i \in J^*} \rangle_{\mathcal{B}} = \mathcal{C}$.
\end{proof}

	Before proving the stability of $(\mathbf{K}_{+}^{3}, \preccurlyeq^+)$, we notice that the monster model $\mathfrak{M}$ of $(\mathbf{K}_{+}^{3}, \preccurlyeq^+)$ can be constructed so that $\mathfrak{M} \in \mathbf{K}_*^{3}$ and furthermore if $\mathcal{A} \in \mathbf{K}^3_+$, $|A| < \kappa$ and $\mathcal{A} \preccurlyeq^+ \mathfrak{M}$, then $\mathcal{A} \preccurlyeq^* \mathfrak{M}$, where $\kappa$ is the cardinal in function of with we construct the monster model (cf. what has been said after Definition \ref{def_AP}). To impose the first condition, we ask that in the construction of $\mathfrak{M}$ we have that $\mathfrak{M} = \bigcup_{i < \lambda} \mathcal{M}_i$ with $\mathcal{M}_i \in \mathbf{K}_*^{3}$ and $\mathcal{M}_i \preccurlyeq^* \mathcal{M}_j$ for $i < j$. To impose the second condition, we ask that for every $\mathcal{A} \in \mathbf{K}^3_+$ such that $|A| < \kappa$ there exists $\mathcal{A}' \preccurlyeq^* \mathcal{M}_0$ such that $\mathcal{A} \cong \mathcal{A}'$ (requirement $(\star)$). This works since if $(\star)$ is satisfied and $\mathcal{A} \in \mathbf{K}^3_+$, $|A| < \kappa$ and $\mathcal{A} \preccurlyeq^+ \mathfrak{M}$, then there exists $\mathcal{A} \cong \mathcal{A}' \preccurlyeq^* \mathcal{M}_0 \preccurlyeq^* \mathfrak{M}$, and so there exists $f \in \mathrm{Aut}(\mathfrak{M})$ sending $\mathcal{A}'$ to $\mathcal{A}$, and thus $\mathcal{A} \preccurlyeq^* \mathfrak{M}$ because $\mathcal{A}' \preccurlyeq^* \mathfrak{M}$. We impose requirement $(\star)$ by defining $\mathcal{M}_0 = \bigcup_{i < \mu} \mathcal{M}_0^i$, where $(\mathcal{M}_0^i)_{i < \mu}$ runs through all the isomorphism types of structures in $\mathbf{K}_{+}^3$ of power $< \kappa$, and we let
	$$\mathcal{M}_0^{i+1} = \mathcal{M}^i_0 \oplus^*_{FG_3} \mathcal{B},$$
where $\mathcal{B} \in \mathbf{K}_*^{3}$ is a witness for $\mathcal{A}_i \in \mathbf{K}_+^3$ (w.o.l.g $FG_3 \preccurlyeq^* \mathcal{M}_i, \mathcal{B})$.

	\begin{theorem}\label{+_is_stable}	$(\mathbf{K}_{+}^{3}, \preccurlyeq^+)$ is stable in every cardinal $\kappa$ such that $\kappa^{\omega} = \kappa$.
\end{theorem}

	\begin{proof} Let $\mathcal{A}$ be infinite and such that $\mathcal{A} \preccurlyeq^+ \mathfrak{M}$, we count $S_1(\mathcal{A})$, i.e. the number of one types over $\mathcal{A}$. For infinite $\mathcal{A}$, the number of one types over $\mathcal{A}$ equals the number of types (of any arity) over $\mathcal{A}$, so it suffices to bound $S_1(\mathcal{A})$. First of all, because of the properties of $\mathfrak{M}$ mentioned above, we have that $\mathcal{A} \preccurlyeq^* \mathfrak{M} \in \mathbf{K}_{*}^{3}$. Let $(a_i, B_i)_{i \in (I, <)}$ be a canonical witness of this, and $a_j, a_k \in \mathfrak{M} - A$. Let then $J^j, J^k \subseteq I$ be closed, countable and such that $j \in J^j$ and $j$ is the largest element of $(J^j, <)$, and similarly for $k$. Thus, by Lemma \ref{closed_implies_strong_subgeo_spec}, we have that 
	$$\mathcal{A} \preccurlyeq^* \langle A, (a_i)_{i \in J^j} \rangle_{\mathfrak{M}} = \mathcal{A}^j \preccurlyeq^* {\mathfrak{M}} \; \text{ and } \; \mathcal{A} \preccurlyeq^* \langle A, (a_i)_{i \in J^k} \rangle_{\mathfrak{M}} = \mathcal{A}^k \preccurlyeq^* {\mathfrak{M}}.$$
Now, if there is an isomorphism $\pi: (J^j, <) \rightarrow (J^k, <)$ such that for all $t \in J^j$ we have that $\pi^*(B_t) = B_{\pi(t)}$, for $\pi^*: \mathcal{A}^j \rightarrow \mathcal{A}^k$ such that $\pi^* \restriction A = \mathrm{id}_A$ and $\pi^*(a_i) = a_{\pi(i)}$ (this determines $\pi^*$ taking images of sups of points as sups of the images of those points), then $\pi^*: \mathcal{A}^j \cong \mathcal{A}^k$, because of Proposition \ref{canonical_determines_str}. Thus, by model homogeneity we have that $\mathrm{tp}(a_j/\mathcal{A}) = \mathrm{tp}(a_k/\mathcal{A})$, since $\mathcal{A}^j, \mathcal{A}^k \preccurlyeq^* \mathfrak{M}$ and being $\pi$ an isomorphism of orders it sends $j$ to $k$. Thus, it suffices to count the number of sets of the form $(a_i, B_i)_{i \in J}$, for $J \subseteq I$ of power $< \omega_1$, modulo isomorphisms $\pi$ as above, and this number is bounded by $|A|^{\omega} \cdot 2^{\omega}$. Hence, if $|A| = \kappa^{\omega} = \kappa$, then $|S_1(\mathcal{A})| = \kappa^{\omega} \cdot 2^{\omega} = \kappa$.
\end{proof}
	
	 Notice that $\mathbf{K}_{+}^{3}$ is closed under ultraproducts (variations on the argument used in Lemma \ref{ultraproduct_lemma}). We then conclude the paper with the following open problem.
	
	\begin{oproblem} Understand the first-order theory of the monster model $\mathfrak{M}(\mathbf{K}_{^+}^{3}, \preccurlyeq^+)$.
\end{oproblem}


\end{document}